\documentclass[10pt]{article}
\usepackage{amsmath, amsfonts, amssymb,  mathrsfs, color, amsthm, epsfig, color, graphics, mathtools,oubraces, enumerate,authblk,esint,empheq,hyperref}
  \usepackage[all]{xy}
 \usepackage[latin1]{inputenc}
\newtheorem{theorem}{Theorem}[section]
\newtheorem*{theorem*}{Theorem}
\newtheorem{lemma}{Lemma}[section]
\newtheorem{prop}{Proposition}[section]

\newtheorem{cor}{Corollary}[section]
\newtheorem{problem}{Problem}
\usepackage{anysize} 
\marginsize{2cm}{2cm}{1.5cm}{1.5cm}
\DeclareMathOperator{\dv}{div}

\DeclareMathOperator{\real}{Re}
\DeclareMathOperator{\imag}{Im}
\DeclareMathAccent{\mpetito}{\mathalpha}{operators}{23}

\newcommand{\ds}{\,{\rm d}s}

\newcommand{\dy}{\,{\rm d}y}

\newcommand{\forallt}{\qquad\text{for all }}

\newcommand{\fct}[4]{\arraycolsep=1.4pt\begin{array}{rcl}{#1}&\longrightarrow&{#2}\\{#3}&\longmapsto&{#4}\end{array}}

\providecommand{\keywords}[1]
{
  \small	
  \textbf{\textit{Keywords---}} #1
}
\title{Reconstruction of obstacles in a Stokes flow as a shape-from-moments problem}
\author{Alexandre Munnier\footnote{\href{mailto:me@somewhere.com}{alexandre.munnier@univ-lorraine.fr}}}
\affil{Université de Lorraine, CNRS, Inria, IECL, F-54000 Nancy, France}
\begin{document}
%
\maketitle
\begin{abstract}
We address the classical inverse problem of recovering the position and shape of obstacles immersed in a planar Stokes flow 
using boundary measurements. We prove that this problem can be transformed into a shape-from-moments problem to which 
ad hoc reconstruction methods can be applied. The effectiveness of this approach is confirmed by numerical tests that show significant improvements over those available in the literature to date. 
\end{abstract}
\section{Introduction}
\label{SEC:1}
A shape-from-moments problem is a geometric inverse problem consisting in recovering the shape of an unknown (possibly 
multi-connected) domain $\mathcal O$ 
from a finite section of its complex moments:
$$\int_{\mathcal O} \overline{z^m} z^n\, {\rm d}m(z)\qquad (n,m\in\mathbb N).$$
In some cases, only the harmonic moments (i.e. for which $m=0$ in the identity above) are available. 
The 2D inverse gravimetric problem is inherently a shape-from-moments problem and can therefore be handled with tools developed for 
this type of problem as it is explained in the recent article \cite{Gerber-Roth:2023aa}. 
More surprisingly, the geometric Calder{\'o}n inverse problem has also been shown in \cite{Munnier:2018aa} to be equivalent to a shape-from-moments problem, leading to an original and efficient reconstruction method.
\par
In this work, we are going to show that the problem of detecting (and reconstructing) obstacles immersed in a Stokes flow 
from boundary measurements can be seen (and dealt with) as a shape-from-moments problem as well. 
Let us go into   more detail and start by making precise the
geometric parameters and the problem under consideration:
\par
Let $\varOmega$ and $\mathcal O_j$ ($j=1,\ldots,N$) be open Jordan domains of class  $\mathcal C^{1,1}$ such 
that $\overline{\mathcal O_j}\subset \varOmega$ and $\overline{\mathcal O_j}\cap\overline{\mathcal O_k}=
\varnothing$ for every indices $j,k=1,\ldots,N$, $j\neq k$. The boundaries of $\varOmega$ and $\mathcal O_j$ are denoted by  $\varGamma_0$ 
and $\varGamma_j$  respectively and we also denote:
$$\mathcal O=\cup_{j=1}^N\mathcal O_j\qquad\text{and}\qquad\varGamma=\cup_{j=1}^N\varGamma_j.$$
 The unit normal vector field $n$ defined on $\varGamma$ and on $\varGamma_0$  is always 
assumed to be pointing towards the interior of the domain enclosed by the Jordan curves (see Fig.~\ref{fig2}).
\begin{figure}[ht]
\centerline{\input{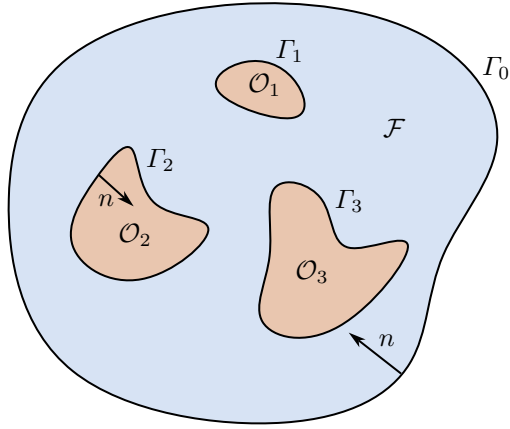}}
\caption{\label{fig2}$\mathcal F$ and $\mathcal O_j$ are respectively the fluid and obstacle domains, and $\varOmega$ is the union 
of $\mathcal F$ and  $\overline{\mathcal O}_j$ ($j=1,\ldots,N$). The boundaries of the domains are $\mathcal C^{1,1}$ Jordan curves. At every point of the boundaries, the unit normal vector $n$ is directed 
towards the interior of the domain enclosed by the curve.}
\end{figure}
We assume that the domain $\mathcal F=
\varOmega\setminus\cup_{j=1}^N \overline{\mathcal O}_j$ is filled with a fluid whose flow is governed by the Stokes equations.
Thus, the velocity and pressure fields $u$ and $p$ verify:
\begin{subequations}
\label{primitives}
\begin{empheq}[left=\empheqlbrace]{align}
\label{eq:stokes_1}
-\nu\Delta u+\nabla p &=0\quad\text{in }\mathcal F,\\
\label{eq:stokes_2}
\dv u&=0\quad\text{in }\mathcal F,\\
\label{eq:stokes_3}
u&=0\quad\text{on }\varGamma,\\
\label{eq:stokes_4}
u&=v\quad\text{on }\varGamma_0,
\end{empheq}
where $\nu>0$ stands for the kinematic viscosity of the fluid and $v$ is a prescribed velocity on $\varGamma_0$ satisfying the flux condition:
\begin{equation}
\label{f_cond}
\int_{\varGamma_0} v\cdot n\ds= 0.
\end{equation}
\end{subequations}
The Cauchy stress tensor is 
defined in $\mathcal F$ by:
\begin{equation}
\label{C_T_tensor}
T(u,p)=\nu\big[\nabla u+(\nabla u)^t\big]-p\,{\rm Id}.
\end{equation}
The classical reconstruction problem we focus on in this article consists in determining the shape and position of the obstacles $\mathcal O_j$ by observing the Cauchy force $T(u,p)n$ along  $\varGamma_0$. More precisely, introducing the Sobolev space:
$$\widetilde H^{1/2}(\varGamma_0;\mathbb R^2)=\Big\{v\in H^{1/2}(\varGamma_0;\mathbb R^2)\,:\, \int_{\varGamma_0} v\cdot n\ds=0\Big\},$$
and denoting by $H^{-1/2}(\varGamma_0;\mathbb R^2)$ its dual space (using $L^2(\varGamma_0;\mathbb R^2)$ as pivot space), we   define 
the operator:
$$\fct{\widetilde\varLambda_\varGamma:\widetilde H^{1/2}(\varGamma_0;\mathbb R^2)}{H^{-1/2}(\varGamma_0;\mathbb R^2)}{v}{T(u,p)n,}$$
where $(u,p)$ is the solution of System \eqref{primitives}. We can now state:
 \begin{problem}
 \label{pb1}
Reconstruct the multi-connected domain $\mathcal O$ from $\widetilde\varLambda_\varGamma$.
 \end{problem}
In \cite{Alvarez:2005aa} it is established that $\widetilde\varLambda_\varGamma$ uniquely determines the multi-connected domain 
$\mathcal O$.
Relevant references on this problem can be found  in 
\cite{Caubet:2016ab}. As far as reconstruction methods are concerned, most are based on the minimization of a cost functional using shape sensitivity analysis techniques.
However, as explained in \cite{Badra:2011aa}, the problem is severely ill-posed and generally these methods, although capable of roughly capturing the position and shape of the obstacles, give mixed results. To achieve a satisfactory reconstruction, the obstacles generally have to 
satisfy certain constraints: they must be close to the outer boundary (where the measurements are made), sometimes sufficiently small, sometimes their number must be known a priori.
\par
The new approach we propose in this paper includes the following steps: first, we will show that Problem~\ref{pb1} can be equivalently reformulated in terms of vorticity and stream function (the so-called ``non-primitive variables''), leading to an inverse problem for the biharmonic operator. 
This problem will then be transformed into a system of integral equations involving biharmonic single-layer potentials.  Following an idea 
initially introduced in \cite{Munnier:2017aa}, we will explain how the formulation can be modified so that only integrals on the obstacles boundaries appear and this will allow us to compute, for any harmonic functions $h_1,h_2$ in $\varOmega$, the quantities:
$$\int_{\mathcal O} h_1 h_2\, {\rm d}m.$$
By choosing harmonic monomials for $h_1$ and $h_2$, we will be able to deduce the  complex moments of the obstacles:
\begin{equation}
\label{gtyhui}
\int_{\mathcal O} \overline{z^m} z^n\, {\rm d}m(z)\forallt n,m\in\mathbb N,
\end{equation}
thus transforming the initial reconstruction problem into a classical shape-from-moments problem. At this point, two algorithms of reconstruction 
can be used. The first one, described in \cite{Gustafsson:2009aa} requires a finite section of 
all the complex moments. The boundaries of the obstacles are 
obtained as the zero level line of a function constructed from the orthogonal polynomials on $\mathcal O$.
The second requires only a finite section of the harmonic moments. The obstacles are first considered as Dirac masses whose 
locations and weights are obtained by solving a so-called Prony's system (a non-linear system associated to a Vandermonde matrix).  Then, the shapes 
of the obstacles are deduced by applying a ``partial-balayage'' operator (derived from the theory of ``balayage'' of measures). This method 
was successfully applied to solve inverse gravimetric problems in \cite{Gerber-Roth:2023aa} and will be seen to provide fairly accurate results
regardless of the size, number and location of the obstacles.
\par
It should be noted that one of the original features of the proposed algorithms is that they are not iterative, unlike most of those available to date.
They require a finite number of measurements and provide a direct approximation of the target domains.
\section{The problem in non-primitive variables}
\label{SEC:2}
For any vector $x=(x_1,x_2)$ in $\mathbb R^2$, $x^\perp$ is the vector rotated counter-clockwise, i.e. $x^\perp=(-x_2,x_1)$. The velocity field $u$ in System \eqref{primitives} satisfies $\dv u=0$ in $\mathcal F$ and 
$\int_{\mathscr C}u\cdot n\ds=0$ for any Jordan curve $\mathscr C$ included in $\mathcal F$. This ensures the existence  of 
a stream function $\psi$ such that $\nabla^\perp\psi=u$ in $\mathcal F$. From 
  \eqref{eq:stokes_1}, we deduce that:
\begin{equation}
\label{holo}
-\nu\nabla^\perp\omega+\nabla p = 0\quad\text{in }\mathcal F,
\end{equation}
where $\omega=\Delta\psi$ is the vorticity field. 
This implies (by taking the rotational of the identity above) that $\omega$ is harmonic in $\mathcal F$ and that:
$$
\int_{\varGamma_j}\partial_n\omega\ds=0\forallt j=1,\ldots,N.
$$
From the prescribed velocity field $v$ in \eqref{eq:stokes_4}, we define a function $f^d$ on $\varGamma_0$ such that $\partial_\tau f^d=v_\cdot n$, 
where $\tau$ is the unit tangent vector field on $\varGamma_0$ oriented such that $\tau^\perp=n$. The existence of $f^d$ is guaranteed by 
the flux condition \eqref{f_cond}. Let also introduce 
$f^n=-v\cdot\tau$ on $\varGamma_0$. The boundary value problem \eqref{primitives} can be restated in 
terms of non-primitive variables $(\psi,\omega)$ as follows:
\begin{subequations}
\label{sys_courant_1}
\begin{empheq}[left=\empheqlbrace]{align}
\label{sys_courant_1_eq1}
\Delta\omega &=0\quad\text{in }\mathcal F,\\
\Delta\psi &=\omega\quad\text{in }\mathcal F,\\
\label{sys_courant_1_eq3}
\big(\psi,\partial_n\psi\big) &=(c_j,0)\quad\text{on }\varGamma_j\text{ for }j=1,\ldots,N,\\
\label{sys_courant_1_eq4}
\big(\psi,\partial_n\psi\big) &=(f^d,f^n)\quad\text{on }\varGamma_0,
\end{empheq}
where the $c_j$ ($j=1,\ldots,N$) are real constants such that:
\begin{equation}
\label{sys_courant_1_eq6}
\int_{\varGamma_j}\partial_n\omega\ds=0\forallt j=1,\ldots,N.
\end{equation}
\end{subequations}
Extending the stream function $\psi$ by $c_j$ in $\mathcal O_j$, it can be assumed to be defined in the whole domain 
$\varOmega$. In a similar way, the vorticity 
field $\omega$ is assumed to be defined in $\varOmega$, extended by $0$ in the obstacles. 
In the sequel, we shall also require the functions $\xi_j$ ($j=0,1,\ldots,N$) solving:
\begin{subequations}
\label{defxi}
\begin{empheq}[left=\empheqlbrace]{align}
\Delta\omega_j&=0\quad\text{in }\mathcal F\\
\Delta\xi_j&=\omega_j\quad\text{in }\mathcal F\\
\big(\xi_j,\partial_n\xi_j\big) &=(1,0)\quad\text{on }\varGamma_j,\\
\big(\xi_j,\partial_n\xi_j\big) &=(0,0)\quad\text{on }\varGamma_k,\quad(k=0,1,\ldots,N,\,k\neq j).
\end{empheq}
\end{subequations}
The proof of the existence and uniqueness of the functions $\xi_j$ in $H^2(\mathcal F)$ is classical (see \cite[Proposition 1.3]{Girault:1986aa}). Once extended by suitable 
contants in the obstacles  they can be considered as functions of the space $H^2_0(\varOmega)$. For every function $u\in H^2(\varOmega)$, we can define the Dirichlet and Neumann traces on any curve $\varGamma_k$ ($k=0,\ldots,N$)  denoted respectively 
by $\gamma_{\varGamma_k}^d u$ 
and $\gamma_{\varGamma_k}^n u$ (the Neumann trace is defined taking into account the orientation of $n$). The total trace operator is next 
given by:
$$\fct{\gamma_{\varGamma_k}:H^2(\varOmega)}{H^{3/2}({\varGamma_k})\times H^{1/2}(\varGamma_k)}
{u}{\big(\gamma_{\varGamma_k}^d u,\gamma_{\varGamma_k}^n u\big).}$$
In the sequel we denote by $H(\varGamma_k)$ the space $H^{3/2}({\varGamma_k})\times H^{1/2}(\varGamma_k)$ and 
by $H'(\varGamma_k)=H^{-3/2}({\varGamma_k})\times H^{-1/2}(\varGamma_k)$ its dual space, using $L^2(\varGamma_k)\times L^2(\varGamma_k)$ as pivot space.
\begin{prop}
\label{PROP:1}
For any $f=(f^d,f^n)\in H(\varGamma_0)$, there exists a unique function $\psi$ in $H^2(\varOmega)$ constant in every $\mathcal O_j$ 
($j=1,\ldots,N$) and such that its restriction to $\mathcal F$ solves  
System \eqref{sys_courant_1}. It is obtained as the unique function achieving:
\begin{equation}
\label{eq:min}
\min\Big\{\|\Delta u\|_{L^2(\varOmega)}\,:\,\gamma_{\varGamma_0}u=f,\,\gamma_{\varGamma_j}u=(c_j,0),\, c_j\in\mathbb R,\,(j=1,\ldots,N),\,u\in H^2(\varOmega)\Big\}.
\end{equation}
\end{prop}
\begin{proof}
Let $u_f$ be a function in $H^2(\varOmega)$ such that $\gamma_{\varGamma_0}u_f=f$ and $u_f=0$ in $\mathcal O$. 
The space $H^2(\varOmega)$ is provided with the scalar product:
$$\big(\theta_1,\theta_2)_{\varOmega}=\big(\Delta \theta_1,\Delta \theta_2)_{L^2(\varOmega)}+
\big(\gamma_{\varGamma_0}\theta_1,\gamma_{\varGamma_0}\theta_2\big)_{H(\varGamma_0)},\qquad\theta_1,\theta_2\in H^2_0(\varOmega),$$
where $(\cdot,\cdot)_{H(\varGamma_0)}$ stands for the scalar product in $H(\varGamma_0)$.
Identifying $H^2_0(\mathcal F)$ with the subspace of $H^2_0(\varOmega)$ consisting in the functions that vanishe in $\mathcal O$, we define 
the subspace of $H^2(\varOmega)$:
$$B(\varOmega)=\big(H^2_0(\mathcal F)\oplus\langle \xi_1,\ldots,\xi_N\rangle\big)^\perp.$$
It is classical to verify that the function achieving \eqref{eq:min} is the orthogonal projection of $u_f$ on $B(\varOmega)$ in $H^2(\varOmega)$ 
and that it is unique. 
\end{proof}
%
From identities \eqref{holo} and \eqref{sys_courant_1_eq1}, we deduce that $\omega$ and $p$ are conjugate harmonic functions, or equivalently that the complex function 
$\nu\omega+ip$ is holomorphic in $\mathcal F$. It means that (up to a constant) $p$ can be deduced from $\omega$ (see \cite[\S 7.1]{Langlois:2014aa}).
We denote by $D^2\psi$ the Hessian matrix of $\psi$ and by $(\perp)$ the rotation matrix of angle $\pi/2$ counter-clockwise. 
The Cauchy stress tensor \eqref{C_T_tensor} reads:
$$T(\psi,\omega)=\nu\big[(\perp)D^2\psi-D^2\psi(\perp)\big]-p\,{\rm Id}.$$
Introducing $s$, the arc length parameterization on $\varGamma_0$, we recall the relations:
$$\partial_s n(s)=-\kappa(s)\tau(s)\qquad\text{and}\qquad\partial_s\tau(s) = \kappa(s) n(s),$$
where the function $\kappa$ is the curvature of $\varGamma_0$. Forming the scalar product of $T(\psi,\omega)n$ with $n$ and $\tau$ on 
$\varGamma_0$ we obtain:
$$
T(\psi,\omega)n\cdot n=2\nu D^2\psi n\cdot \tau-p\qquad\text{and}\qquad
T(\psi,\omega)n\cdot \tau=\nu\big[-D^2\psi\tau\cdot\tau+D^2\psi n\cdot n\big].$$
Using the identities $\partial_s(\nabla\psi\cdot n)=D^2\psi \tau\cdot n -\kappa\partial_\tau \psi$, 
$\omega = D^2\psi\tau\cdot\tau+D^2\psi n\cdot n$ and
$\partial_s(\nabla\psi\cdot\tau)=D^2\psi\tau\cdot\tau+\kappa\partial_n\psi$
we end up with:
$$T(\psi,\omega)n\cdot n=2\nu\, \kappa\, \partial_\tau \psi+2\,\nu\,\partial_s(\partial_n\psi)-p
\qquad\text{and}\qquad T(\psi,\omega)n\cdot \tau=-2\nu\,\kappa\, \partial_n\psi+2\nu\,\partial_s(\partial_\tau\psi)-\nu\,\omega.$$
Because of \eqref{holo}, we have on $\varGamma_0$, $\partial_sp=-\nu\partial_n\omega$, and hence $p$ is an antiderivative of 
$-\nu\partial_n\omega$ 
on $\varGamma_0$ so we denote $p=-\nu\int \partial_n\omega$. Summarizing, taking into account \eqref{sys_courant_1_eq4}, we 
obtain eventually:
$$
T(\psi,\omega)n\cdot n=\nu\Big[2  \kappa \partial_s f^d+2 \partial_s f^n+ \int \partial_n\omega\Big]\qquad\text{and}\qquad
T(\psi,\omega)n\cdot \tau=\nu\Big[2 \partial_{ss}^2 f^d-2 \kappa f^n-\omega\Big].
$$
So, for prescribed $f^d$ and $f^n$, it is equivalent to measure the Cauchy force exerted by the fluid on $\varGamma_0$ and 
to measure the pair $(\omega,\partial_n\omega)$ on this boundary. The function $\omega$ in System \eqref{sys_courant_1} is
harmonic and square-integrable on $\varOmega$, so it admits a Dirichlet and a Neumann trace on $\varGamma_0$ 
and we can define:
\begin{equation}
\label{fct_Lambda}
\fct{\varLambda_\varGamma:H(\varGamma_0)}{H'(\varGamma_0)}
{(f^d,f^n)}{(\partial_n\omega|_{\varGamma_0},-\omega|_{\varGamma_0}).}
\end{equation}
From now on, we shall focus on the following problem:
\begin{problem}
\label{pb2}
Reconstruct the multi-connected domain $\mathcal O$ from $\varLambda_\varGamma$.
\end{problem}
 \section{Identifiability for the biharmonic inverse problem}
Identifiability results for Problem~\ref{pb1} are given in \cite[Theorem 3.1]{Alvarez:2005aa} and in \cite[Theorem 2.1]{Alves:2007aa} and 
 can probably be adapted to Problem~\ref{pb2}. We prefer to provide an original proof directly for the biharmonic inverse problem.
\begin{theorem}
\label{theo:ident}
Let $\mathcal O$ and $\mathcal O'$ be two sets of obstacles as described in Section~\ref{SEC:1}. 
Assume that there exists $f=(f^d,f^n)\in H(\varGamma_0)$ and an open set $\mathscr U$ 
in $\mathbb R^2$ such that $\mathscr U\cap \varGamma_0\neq \varnothing$, $ {\mathscr U}\cap  \big({\mathcal O\cup\mathcal O'}\big)=
\varnothing$, $f^d$ non-constant in $\mathscr U\cap \varGamma_0$ and:
$$ \langle\varLambda_\varGamma f,\gamma_{\varGamma_0}\theta\rangle_{\varGamma_0}=\langle\varLambda_{\varGamma'} f,\gamma_{\varGamma_0}\theta\rangle_{\varGamma_0}\forallt \theta\in\mathscr D(\mathscr U).$$
Then 
$\mathcal O=\mathcal O'$. 
\end{theorem}
In this statement, the bracket $\langle\cdot,\cdot\rangle_{\varGamma_0}$ stands for the duality pairing in 
$H'(\varGamma_0)\times H(\varGamma_0)$ that extends the scalar product of $L^2(\varGamma_0)\times L^2(\varGamma_0)$.
\begin{proof}Denote by $\psi$ and $\psi'$ the solutions to System \eqref{sys_courant_1} corresponding respectively to the
obstacles $\mathcal O$ and 
$\mathcal O'$ and  the same right hand side $f=(f^d,f^n)$ in identity \eqref{sys_courant_1_eq4}. As already explained earlier, these functions extended by suitable constants in the obstacles are in $H^2(\varOmega)$.
The function $\varphi=\psi-\psi'$ is furthermore extended by zero outside $\varOmega$ and 
since $\gamma_{\varGamma_0}\varphi=0$, this function is in $H^2(\mathbb R^2)$ and is biharmonic in $\mathbb R^2\setminus\big(\varGamma_0\cup \varGamma\cup\varGamma'\big)$. For every $\theta\in\mathscr D(\mathscr U)$, an integration by parts 
yields:
$$\big(\Delta\varphi,\Delta\theta\big)_{L^2(\mathbb R^2)}=\langle\varLambda_{\varGamma} f-\varLambda_{\varGamma'} f,\gamma_{\varGamma_0}\theta\rangle_{\varGamma_0}=0,$$
which means that $\Delta\varphi$ is harmonic in $\mathscr U$. But since $\varphi=0$ in the open set 
$\mathscr U\setminus\overline{\varOmega}$, the unique continuation principle entails 
 that the function $\varphi$  vanishes in the whole set $\mathbb R^2\setminus\big(\overline{\mathcal O\cup\mathcal O'}\big)$. 
\par
For every  $j=1,\ldots,N$, we can decompose the boundary $\varGamma_j$ into two parts:
$$\varGamma_j=\big(\varGamma_j\cap\overline{\mathcal O'}\big)\cup\big(\varGamma_j\cap(\mathbb R^2\setminus\overline{\mathcal O'})\big).$$
On $\varGamma_j\cap\overline{\mathcal O'}$, $\psi'$ is piecewise constant by definition and this is also true on $\varGamma_j\cap(\mathbb R^2\setminus\overline{\mathcal O'})$ because $\psi'=\psi$ on this set. Since $\psi'$ is continuous, $\psi'$ is constant on every 
curve $\varGamma_j$. 
According the Proposition~\ref{PROP:1}, this implies that $\|\Delta\psi'\|_{L^2(\varOmega)}\leqslant \|\Delta\psi\|_{L^2(\varOmega)}$. The role 
played by $\psi$ and $\psi'$ being symmetric, it turns out that $\|\Delta\psi'\|_{L^2(\varOmega)}= \|\Delta\psi\|_{L^2(\varOmega)}$ and 
then $\psi=\psi'$ by uniqueness of the minimum in Proposition~\ref{PROP:1}. If there were an index $j\in\{1,\ldots,N\}$ such that the open set $\mathcal O'_j\setminus\overline{\mathcal O}$ were non-empty, the function $\psi$ would be constant on this set (because $\psi'$ is) and therefore constant on the whole set $\varOmega\setminus\overline{\mathcal O}$. This would imply that $f^d$ is constant, which is excluded by 
hypothesis. We deduce that 
$\mathcal O\subset \mathcal O'$ and, because $\mathcal O$ and $\mathcal O'$ play symmetric roles, that $\mathcal O=\mathcal O'$.
\end{proof}
 \section{Biharmonic single-layer potential}
 This section is devoted to establishing (and recalling) some results on the biharmonic single-layer potential.  So we leave aside for a moment the inverse problem 
 we are dealing with and we consider a general framework in which $\varGamma$ represents a finite, disjoint union of $\mathcal C^{1,1}$ 
 Jordan curves. Sticking to our convention, the unit normal vector $n$ on $\varGamma$ is directed towards the interior of the domain 
 enclosed by the curves. If $u$ is a function defined on both sides of $\varGamma$ and admitting one-sided Dirichlet and Neumann 
 traces on $\varGamma$, we can define the jump of these traces across  the curve:
 $$\big[\partial_n u\big]_\varGamma=\gamma_\varGamma^n u^+-\gamma_\varGamma^n u^-\qquad\text{and}\qquad
  \big[ u\big]_\varGamma=\gamma_\varGamma^d u^+-\gamma_\varGamma^d u^-.$$
 These quantities are defined piecewise on each connected component of $\varGamma$ and each component is by definition a Jordan curve. 
The notation $u^-$ represents the restriction of $u$ to the domain  bounded by the curve, while $u^+$ designates the part of the function 
$u$   outside the curve. 
 \par
 The fundamental solution of the Bilaplacian is defined by:
\begin{equation}
\label{eq:def_G}
G(x)=\frac{1}{8\pi}\Big[|x|^2\ln\frac{|x|}{\kappa_0}+\kappa_1\Big]\forallt x\in\mathbb R^2,
\end{equation}
where $\kappa_0$ and $\kappa_1$ are real constants that will be fixed later on. Using the usual abuse of notation to identify $G(x-y)$ with 
a two-variables function $G(x,y)$, the biharmonic single-layer potential  is defined for every $q=(q_n,q_d)\in H'(\varGamma)$  by:
\begin{equation}
\label{eq:def_Sq}
\mathscr S_\varGamma q(x)=\int_\varGamma G (x,y)q_n(y)+\partial_{n(y)}G(x,y)q_d(y)\,{\rm d}s(y)\forallt x\in\mathbb R^2.
\end{equation}
The operator $\mathscr S_\varGamma:H'(\varGamma)\longrightarrow H^2_{\ell oc}(\mathbb R^2)$ is bounded
so the same conclusion applies to the operator:
\begin{equation}
\label{def_Vgamma}
\fct{S_\varGamma:H'(\varGamma)}{H(\varGamma)}
{q}{\gamma_\varGamma\circ\mathscr S_\varGamma q.}
\end{equation}
The results presented without proof in this section are borrowed from \cite{Munnier:2023aa}:
\begin{theorem}
\label{theo:def_pos}
Let $R>0$ be the radius of a circle $\mathcal C_R$ that enclosed $\varGamma$. 
If we choose $\kappa_0>eR$ and $\kappa_1>R^2$ in the definition \eqref{eq:def_G}, then the
operator $S_\varGamma$ is strongly elliptic on $H'(\varGamma)$ and therefore invertible.
\end{theorem}
The so-called jump relations allow recovering $q=(q_n,q_d)$ from the function $\mathscr S_\varGamma q$:
\begin{subequations}
\label{eq:jump}
\begin{align}
q_n&=-\big[\partial_n\Delta \mathscr S_\varGamma q\big]_\varGamma=
-\big(\partial_n \Delta \mathscr S^+_\varGamma q-\partial_n \Delta \mathscr S^-_\varGamma q\big)\\
q_d&=\big[\Delta \mathscr S_\varGamma q\big]_\varGamma=\Delta \mathscr S^+_\varGamma q-\Delta \mathscr S^-_\varGamma q.
\end{align}
\end{subequations}
Under the hypotheses of Theorem~\ref{theo:def_pos}, the spaces $H'(\varGamma)$ and $H(\varGamma)$ can be provided 
with the scalar products:
\begin{alignat*}{3}
\big(q,q')_{H'(\varGamma)}&=\big\langle q,S_\varGamma q'\big\rangle_\varGamma,&&\forallt q,q'\in H'(\varGamma),\\
\big(p,p')_{H(\varGamma)}&=\big\langle S_\varGamma^{-1}p,p'\big\rangle_\varGamma,&&\forallt p,p'\in H(\varGamma),
\end{alignat*}
where $\langle\cdot,\cdot\rangle_\varGamma$ stands for the duality pairing on $H'(\varGamma)\times H(\varGamma)$ that extends the scalar product 
on $L^2(\varGamma)\times L^2(\varGamma)$.
All the inclusions below are continuous and dense:
$$H(\varGamma)\subset L^2(\varGamma)\times L^2(\varGamma)\subset H'(\varGamma),$$
and $S_\varGamma$ is an isometric operator from $H'(\varGamma)$ onto $H(\varGamma)$. In the following, for every $p\in H(\varGamma)$, 
we denote $\widehat p=S_\varGamma^{-1}p$  so that $\gamma_\varGamma\circ \mathscr S_\varGamma\widehat p=p$ and 
$\|p\|_{H(\varGamma)}=\|\widehat p\|_{H'(\varGamma)}$.
\begin{lemma}
Let $\mathscr A$ be the three-dimensional subspace of $H(\varGamma)$ spanned by the total traces of the affine functions in $\mathbb R^2$.
Then, for every $p, p'\in \mathscr A^\perp$, $\mathscr S_\varGamma \widehat p$ and $\mathscr S_\varGamma \widehat p'$ are 
in $L^2(\mathbb R^2)$ and:
\begin{equation}
\label{eq:zero_fct}
\big(p,p')_{H(\varGamma)}=\int_{\mathbb R^2}\Delta \big(\mathscr S_\varGamma \widehat p\big)\,\Delta  \big(\mathscr S_\varGamma \widehat p'\big){\rm d}m.
\end{equation}
\end{lemma}
\begin{lemma}
\label{lem:inter}
Let $\varGamma$ and $\varGamma'$ be two sets of non-intersecting $\mathcal C^{1,1}$ Jordan curves. Then for every $p_\varGamma\in H(\varGamma)$ 
and $p_{\varGamma'}\in H(\varGamma')$, we have:
$$\big(\gamma_{\varGamma}\circ\mathscr S_{\varGamma'}\widehat p_{\varGamma'},S_\varGamma \widehat p_{\varGamma}\big)_{H(\varGamma)}=
\big(S_{\varGamma'} \widehat p_{\varGamma'},\gamma_{\varGamma'}\circ\mathscr S_{\varGamma}\widehat p_{\varGamma}\big)_{H(\varGamma')}.$$
\end{lemma}
\begin{proof}
Let $\widehat p_{\varGamma}=(\widehat p_{\varGamma}^n,\widehat p_{\varGamma}^d)\in L^2(\varGamma)\times L^2(\varGamma)$ and $\widehat p_{\varGamma'}=(\widehat p_{\varGamma'}^n,\widehat p_{\varGamma'}^d)\in L^2(\varGamma')\times L^2(\varGamma')$. Then:
\begin{multline*}
\big(\gamma_{\varGamma}\circ\mathscr S_{\varGamma'}\widehat p_{\varGamma'},S_\varGamma \widehat p_{\varGamma}\big)_{H(\varGamma)}
=\int_{\varGamma}\left(\int_{\varGamma'}G(x,y)\widehat p^n_{\varGamma'}(y)+\partial_{n(y)}G(x,y)\widehat p_{\varGamma'}^d(y)\ds(y)\right)\widehat p_{\varGamma}^n(x) 
\\
+\left(\int_{\varGamma'}\partial_{n(x)}G(x,y)\widehat p^n_{\varGamma'}(y)+\partial_{n(x)}\partial_{n(y)}G(x,y)\widehat p_{\varGamma'}^d(y)\dy\right)\widehat p_{\varGamma}^d(x)\ds(x).
\end{multline*}
Since $\varGamma$ and $\varGamma'$ do not intersect, neither kernel is singular. So, we can reverse the order of integration
and conclude with a density argument.
\end{proof}
 \section{The reconstruction problem as a shape-from-moments problem}
 We return now to  Problem~\ref{pb2} with the notation introduced in Section~\ref{SEC:2}.
In Definition~\ref{eq:def_G}, we choose the constants $\kappa_0$ and $\kappa_1$ in such a way that the operators 
$S_{\varGamma_0}$ and $S_\varGamma$ be strongly elliptic as explained in Theorem~\ref{theo:def_pos}. The solution $\psi$ to System~\ref{sys_courant_1} can therefore be represented as a sum of biharmonic single-layer potentials:
$$\psi=\mathscr S_{\varGamma_0} \widehat p_{\varGamma_0}+\mathscr S_\varGamma \widehat p_\varGamma.$$
The densities $\widehat p_{\varGamma_0}=(\widehat p_{\varGamma_0}^n,\widehat p_{\varGamma_0}^d)\in H'(\varGamma_0)$ and $\widehat p_\varGamma
=(\widehat p_{\varGamma}^n,\widehat p_{\varGamma}^d)\in H'(\varGamma)$ 
satisfy the system:
\begin{subequations}
\label{sys:K}
\begin{empheq}[left=\empheqlbrace]{align}
\label{sys:Ka}
S_{\varGamma_0}\widehat p_{\varGamma_0}+\gamma_{\varGamma_0}\circ\mathscr S_{\varGamma}\widehat p_\varGamma&=f \quad \text{in }H(\varGamma_0)\\
\label{sys:Kb}
\gamma_{\varGamma}\circ\mathscr S_{\varGamma_0}\widehat p_{\varGamma_0}+S_\varGamma \widehat p_\varGamma&=(c,0)\quad \text{in }H(\varGamma).
\end{empheq}
In the second identity, $c$ is a function  equal to a real constant $c_j$ on each connected component $\varGamma_j$ ($j=1,\ldots,N$) of $\varGamma$. These constants are uniquely determined by the conditions:
\begin{equation}
\label{eq:nf}
\int_{\varGamma_j}\widehat p^n_\varGamma\ds=0,\qquad(j=1,\ldots,N).
\end{equation}
\end{subequations}
For every $j=1,\ldots,N$, let  $\mathbf 1_{\varGamma_j}$ be the piecewise constant function in $H(\varGamma)$, equal to $(1,0)$ on $\varGamma_j$ 
and $(0,0)$ on $\varGamma_k$ for $k\neq j$ and define 
the subspaces of $H'(\varGamma)$ and $H(\varGamma)$ of codimension $N$:
$${\mathcal H}'(\varGamma)=\big\{q\in H'(\varGamma)\,:\, \langle q,\mathbf 1_{\varGamma_j}\rangle_\varGamma=0,\quad\forall\, j=1,\ldots,N\big\}\quad\text{and}\quad
{\mathcal H}(\varGamma)=\big\{p\in H(\varGamma)\,:\, \langle \widehat{\mathbf 1}_
{\varGamma_j},p\rangle=0,\quad\forall\, j=1,\ldots,N\big\}.$$
Notice that identities \eqref{eq:nf} mean that $\widehat p_\varGamma$ belongs to ${\mathcal H}'(\varGamma)$ and that 
the operator $S_\varGamma$ isometrically maps ${\mathcal H}'(\varGamma)$ onto ${\mathcal H}(\varGamma)$. We introduce the operators:
\begin{equation}
\label{def_KGG}
\fct{K_{\varGamma_0}^{\varGamma}:H(\varGamma_0)}{H(\varGamma)}{p}{\gamma_{\varGamma}\circ \mathscr S_{\varGamma_0}\widehat p}
\qquad\text{and}\qquad
\fct{K_{\varGamma}^{\varGamma_0}:H(\varGamma)}{H(\varGamma_0)}{p}{\gamma_{\varGamma_0}\circ \mathscr S_{\varGamma}\widehat p,}
\end{equation}
and we deduce from Lemma~\ref{lem:inter} that:
\begin{equation}
\label{PRO:41}
\big(K_{\varGamma_0}^\varGamma p_{\varGamma_0},p_\varGamma\big)_{H(\varGamma)}=
\big(p_{\varGamma_0},K_{\varGamma}^{\varGamma_0} p_{\varGamma}\big)_{H(\varGamma_0)}\forallt 
p_\varGamma\in H(\varGamma),\,p_{\varGamma_0}\in H(\varGamma_0).
\end{equation}
We denote by 
$\Pi_{\varGamma}$ 
the orthogonal projection from $H(\varGamma)$ onto ${\mathcal H}(\varGamma)$ and we apply this operator to equation \eqref{sys:Kb}.
Rewriting System \eqref{sys:K} in terms of traces instead of densities
we obtain:
\begin{subequations}
\label{sys:K1}
\begin{empheq}[left=\empheqlbrace]{align}
\label{sys:K1a}
p_{\varGamma_0}+K_\varGamma^{\varGamma_0}   p_\varGamma&=f \quad \text{in }H(\varGamma_0)\\
\Pi_\varGamma\circ K_{\varGamma_0}^\varGamma p_{\varGamma_0}+p_\varGamma&=0\quad\text{in }{\mathcal H}(\varGamma).
\end{empheq}
\end{subequations}
From the identity \eqref{PRO:41}, specifying that $p_{\varGamma_0}$ is equal to $\mathbf 1_{\varGamma_0}$, we deduce that if 
$p_\varGamma$ is in ${\mathcal H}(\varGamma)$, then 
$K_\varGamma^{\varGamma_0}   p_\varGamma$ is in ${\mathcal H}(\varGamma_0)$. 
It follows from equation \eqref{sys:K1a} that if $f$ belongs to ${\mathcal H}(\varGamma_0)$, $p_{\varGamma_0}$ belongs to 
${\mathcal H}(\varGamma_0)$ as well. We conclude that for every 
$f\in{\mathcal H}(\varGamma_0)$, there exists $(p_{\varGamma_0},  p_{\varGamma})\in {\mathcal H}(\varGamma_0)
\times {\mathcal H}(\varGamma)$ such that:
\begin{subequations}
\label{sys:K2}
\begin{empheq}[left=\empheqlbrace]{align}
\label{sys:K2a}
p_{\varGamma_0}+K_\varGamma^{\varGamma_0}   p_\varGamma&=  f \quad \text{in }{\mathcal H}(\varGamma_0)\\
\label{sys:K2b}
\Pi_\varGamma\circ K_{\varGamma_0}^\varGamma   p_{\varGamma_0}+   p_\varGamma&=0 \quad\text{in }{\mathcal H}(\varGamma).
\end{empheq}
\end{subequations}
We now rewrite the measurement function \eqref{fct_Lambda} in terms of biharmonic single-layer potentials:
\begin{subequations}
\label{def_LAMB}
\begin{equation}
\fct{\varLambda_\varGamma:{\mathcal H}(\varGamma_0)}{{\mathcal H}'(\varGamma_0)}{  f}
{\Big(
\gamma^n_{\varGamma_0}\big(\Delta\mathscr S^+_{\varGamma}  \widehat p_{\varGamma}+\Delta \mathscr S^-_{\varGamma_0}  \widehat p_{\varGamma_0}\big)
,
-\gamma^d_{\varGamma_0}\big(\Delta\mathscr S^+_{\varGamma}  \widehat p_{\varGamma}+\Delta \mathscr S^-_{\varGamma_0}  \widehat p_{\varGamma_0}\big)
\Big),}
\end{equation}
where $(p_{\varGamma_0},p_\varGamma)$ is  the solution of  System~\eqref{sys:K2}. Defining also:
\begin{equation}
\label{def_lambda0}
\fct{\varLambda_0:{\mathcal H}(\varGamma_0)}{{\mathcal H}'(\varGamma_0)}{ f}
{\Big(
\gamma^n_{\varGamma_0}\big(\Delta\mathscr S^-_{\varGamma_0}\circ \widehat f\big),
-\gamma^d_{\varGamma_0}\big(\Delta\mathscr S^-_{\varGamma_0}\widehat  f\big)
\Big),}
\end{equation}
\end{subequations}
and using the identity $\mathscr S^+_{\varGamma_0}\widehat  f=\mathscr S^+_{\varGamma}  \widehat p_{\varGamma}+
\mathscr S^+_{\varGamma_0} \widehat p_{\varGamma_0}$, we first establish that, for every $f\in\mathcal H(\varGamma_0)$:
$$\big(\varLambda_\varGamma-\varLambda_0\big)f=
\Big(
-\big[\partial_n\Delta\mathscr S_{\varGamma_0}\widehat p_{\varGamma_0}\big]_{\varGamma_0}+\big[\partial_n\Delta\mathscr S_{\varGamma_0}\widehat f\big]_{\varGamma_0}
,
\big[\Delta\mathscr S_{\varGamma_0}\widehat p_{\varGamma_0}\big]_{\varGamma_0}-\big[\Delta\mathscr S_{\varGamma_0}\widehat f\big]_{\varGamma_0}
\Big).$$
Then, from the relations \eqref{eq:jump}, we deduce that 
$S_{\varGamma_0}\circ \big(\varLambda_\varGamma-\varLambda_0\big)f=p_{\varGamma_0}-f$
and therefore, denoting by $R_\varGamma$ the operator 
$S_{\varGamma_0}\circ \big(\varLambda_\varGamma-\varLambda_0\big)$, we finally obtain that:
\begin{subequations}
\begin{equation}
\label{eq:inv1}
\big({\rm Id}+R_\varGamma)f=p_{\varGamma_0}.
\end{equation}
On the other hand, 
according to \eqref{sys:K2a}, we have $p_{\varGamma_0}=f-K_\varGamma^{\varGamma_0}   p_\varGamma$
and applying the operator $K_{\varGamma}^{\varGamma_0}$ to \eqref{sys:K2b} yields 
$K_{\varGamma}^{\varGamma_0}  p_\varGamma=-K_{\varGamma}^{\varGamma_0}\circ\Pi_\varGamma\circ K_{\varGamma_0}^\varGamma   p_{\varGamma_0}$.
Therefore, denoting $K_\varGamma=K_{\varGamma}^{\varGamma_0} \circ\Pi_\varGamma\circ K_{\varGamma_0}^\varGamma$ we get:
\begin{equation}
\label{eq:inv2}
\big({\rm Id}-K_\varGamma\big)p_{\varGamma_0}=f.
\end{equation}
\end{subequations}
The theorem below follows directly from the identities \eqref{eq:inv1} and \eqref{eq:inv2}:
%
\begin{theorem}
\label{theo:fact}
The operators $\big({\rm Id}+R_\varGamma):\mathcal H(\varGamma_0)\longrightarrow\mathcal H(\varGamma_0)$ and $\big({\rm Id}-K_\varGamma\big):\mathcal H(\varGamma_0)\longrightarrow\mathcal H(\varGamma_0)$ are invertible and they are each other's inverse.
Moreover:
\begin{equation}
\label{eq:factor}
K_\varGamma=R_\varGamma\circ\big({\rm Id}+R_\varGamma\big)^{-1}.
\end{equation}
\end{theorem}
From the point of view of solving the inverse problem we are dealing with, it is important to note that the right-hand member in the identity \eqref{eq:factor} 
can be computed from $\varLambda_\varGamma$ and $\varLambda_0$, two operators we assume to be known. We now turn our attention to the  operator $K_\varGamma$.
From equality \eqref{PRO:41}, we deduce straightforwardly:
\begin{theorem}
\label{Theo:52}
For every $f,g\in \mathcal H(\varGamma_0)$:
\begin{equation}
\label{eq:keystone}
\big(K_\varGamma f,g\big)_{H(\varGamma_0)}=\Big(  \Pi_\varGamma\circ K_{\varGamma_0}^\varGamma f,
K_{\varGamma_0}^\varGamma g\Big)_{H(\varGamma)}.
\end{equation}
\end{theorem}
This theorem is the keystone of the reconstruction method. Indeed, 
let $F$ et $G$ be biharmonic functions in $H^2(\varOmega)$, denote $f_{\varGamma_0}=\gamma_{\varGamma_0} F$ and 
$g_{\varGamma_0}=\gamma_{\varGamma_0} G$ and assume that $f_{\varGamma_0}$ and $g_{\varGamma_0}$ 
are in $\mathcal H(\varGamma_0)$ (simply add   an appropriate constant to $F$ and $G$ to satisfy this assumption). Since $\mathscr S_{\varGamma_0}\widehat f_{\varGamma_0}=F$ and $\mathscr S_{\varGamma_0}\widehat g_{\varGamma_0}=G$ in $\varOmega$, we deduce from \eqref{eq:keystone} that:
\begin{equation}
\big(K_\varGamma f_{\varGamma_0},g_{\varGamma_0}\big)_{H(\varGamma_0)}=
\big(\Pi_\varGamma f_{\varGamma},g_{\varGamma}\big)_{H(\varGamma)},
\end{equation}
where $f_\varGamma=\gamma_\varGamma F$ and $g_\varGamma=\gamma_\varGamma G$.  In other words,   we have access  
to the scalar product in $\mathcal H(\varGamma)$ of the traces (up to an additive constant because of the projector $\Pi_\varGamma$) of any two functions $F$ and $G$ that are biharmonic in $\varOmega$. This
 will enable us to calculate the complex moments of $\mathcal O$.
\par
Let us denote by $\mathscr H(\varOmega)$ the space of the harmonic functions in $H^2(\varOmega)$.
\begin{lemma}
\label{LEM:51}
Let $p$ be in $\mathcal H(\varGamma)$ such that $\big(K_{\varGamma}^{\varGamma_0}p,\gamma_{\varGamma_0}h\big)_{H(\varGamma_0)}=0$ 
for every function $h$ in $\mathscr H(\varOmega)$. Then $\mathscr S_{\varGamma}\widehat p$ is harmonic outside $\mathcal O$.
\end{lemma}
\begin{proof}Let $D_\varOmega$ stands for the Dirichlet-to-Neumann operator on $\varGamma_0$ (see Section~\ref{SEC:DtN} in the Appendix), 
let $p$ be in $\mathcal H(\varGamma)$ as in the statement of the Lemma and denote $g_{\varGamma_0}=K_{\varGamma}^{\varGamma_0}p$. 
The hypothesis means that:
$$-\big\langle \big[\partial_n \Delta\mathscr S_{\varGamma_0}\widehat g_{\varGamma_0}\big]_{\varGamma_0},q\big\rangle_{-\frac32,\frac32}+
\big\langle \big[ \Delta\mathscr S_{\varGamma_0}\widehat g_{\varGamma_0}\big]_{\varGamma_0},D_\varOmega q\big\rangle_{-\frac12,\frac12}
=0\forallt q\in H^{3/2}(\varGamma_0).$$
The operator $D_\varOmega$ being self-adjoint (see Proposition~\ref{prop:B1}), it follows that:
$$\big\langle \gamma_{\varGamma_0}^n \Delta\mathscr S^-_{\varGamma_0}\widehat g_{\varGamma_0}-D_\varOmega 
\Delta\mathscr S^-_{\varGamma_0}\widehat g_{\varGamma_0},q\big\rangle_{-\frac32,\frac32}
-\big\langle \gamma_{\varGamma_0}^n
\Delta\mathscr S^+_{\varGamma_0}\widehat g_{\varGamma_0}-D_\varOmega\Delta\mathscr S^+_{\varGamma_0}\widehat g_{\varGamma_0},
q\big\rangle_{-\frac12,\frac12}=0\forallt q\in H^{3/2}(\varGamma_0).$$
 The first term vanishes by definition of $D_\varOmega$, which allows us to deduce that  $\gamma_{\varGamma_0}^n
\Delta\mathscr S^+_{\varGamma_0}\widehat g_{\varGamma_0}=D_\varOmega\Delta\mathscr S^+_{\varGamma_0}\widehat g_{\varGamma_0}$.
Let $u_\varOmega$ be the function in $L^2(\varOmega)$, harmonic and such that $\gamma_{\varGamma_0}^d u_\varOmega=\gamma_{\varGamma_0}^d 
\Delta\mathscr S^+_{\varGamma_0}\widehat g_{\varGamma_0}$ (existence and uniqueness of such a function 
is asserted in Proposition~\ref{PROP:L2H}). The function $u$ defined by $u=u_\varOmega$ 
in $\varOmega$ and $u=\Delta\mathscr S^+_{\varGamma_0}\widehat g_{\varGamma_0}$ in $\mathbb R^2\setminus\overline{\varOmega}$ 
is therefore harmonic in $\mathbb R^2$. According to \cite[Lemma 4.1]{Munnier:2023aa}, $u(x)=\mathscr O(\ln|x|)$ as 
$|x|$ goes to $+\infty$, which together with \cite[Theorem 9.10]{Axler:2001aa} (generalized Liouville Theorem) implies 
that $u$ is constant in $\mathbb R^2$. Referring again to \cite[Lemma 4.1]{Munnier:2023aa}, the only possible constant is zero. Since 
$\mathscr S^+_\varGamma p=\mathscr S^+_{\varGamma_0}\widehat g_{\varGamma_0}$ in 
$\mathbb R^2\setminus\overline{\varOmega}$, we conclude that $\Delta \mathscr S^+_\varGamma  p=0$ in $\mathbb R^2
\setminus\overline{\mathcal O}$. 
\end{proof}
\begin{theorem}
Let $h$  be a harmonic function in $L^2(\varOmega)$. For every $\varepsilon>0$ there exists $f_{\varGamma_0}^\varepsilon$   in $\mathcal H(\varGamma_0)$
such that
$$\left|\big(K_\varGamma f_{\varGamma_0}^\varepsilon,f_{\varGamma_0}^\varepsilon\big)_{H(\varGamma_0)}- 
\int_{\mathcal O} h^2{\rm d}m\right|<\varepsilon.$$
\end{theorem}
\begin{proof}
Let $g_{\varGamma_0}$ be in $ H(\varGamma_0)$ such that $\Delta \mathscr S^-_{\varGamma_0}\widehat g_{\varGamma_0}=h$ (for instance choose $g_{\varGamma_0}=\gamma_{\varGamma_0}u$ with $u$ such that $\Delta u=h$ in $H^2(\varOmega)\cap H^1_0(\varOmega)$). 
Let $g^h_\varGamma$ be the orthogonal projection  in $H(\varGamma)$ of $K_{\varGamma_0}^\varGamma g_{\varGamma_0}$ onto 
 $\overline{K_{\varGamma_0}^\varGamma\mathscr H(\varOmega)}$. Thus, for every $\varepsilon>0$, there 
exists $h^\varepsilon\in\mathscr H(\varOmega)$ such that 
\begin{equation}
\label{eq:est}
\|g^h_\varGamma-K_{\varGamma_0}^\varGamma \circ \gamma_{\varGamma_0}h^\varepsilon\|_{H(\varGamma_0)}<
\varepsilon\big[{1+2\big\|K_{\varGamma_0}^\varGamma g_{\varGamma_0}-g^h_\varGamma\big\|_{H(\varGamma)}}\big]^{-1}.
\end{equation}
Define $f_{\varGamma_0}^\varepsilon=\Pi_{\varGamma_0}\big(g_{\varGamma_0}-\gamma_{\varGamma_0}h^\varepsilon\big)$ where
$\Pi_{\varGamma_0}$ is the orthogonal projection onto $\mathcal H(\varGamma_0)$ in $H(\varGamma_0)$. We have:
$$\big(K_\varGamma f_{\varGamma_0}^\varepsilon,f_{\varGamma_0}^\varepsilon\big)_{H(\varGamma_0)}
=\big(\Pi_\varGamma\circ K^\varGamma_{\varGamma_0} f_{\varGamma_0}^\varepsilon,K^\varGamma_{\varGamma_0}f_{\varGamma_0}^\varepsilon\big)_{H(\varGamma)}=\big(\Pi_\varGamma\circ K^\varGamma_{\varGamma_0} f_{\varGamma_0}^\varepsilon,\Pi_\varGamma\circ K^\varGamma_{\varGamma_0}f_{\varGamma_0}^\varepsilon\big)_{H(\varGamma)}.$$
Notice that, for every $p\in H(\varGamma_0)$, $\Pi_\varGamma\circ K_{\varGamma_0}^\varGamma\circ\Pi_{\varGamma_0}p=
\Pi_\varGamma\circ K_{\varGamma_0}^\varGamma p$ because the functions $\mathscr S_{\varGamma_0}p$ and $\mathscr S_{\varGamma_0}\circ \Pi_{\varGamma_0}p$ differ only up to a constant in $\varOmega$. This implies that:
\begin{subequations}
\label{eq:concl}
\begin{equation}
\big(K_\varGamma f_{\varGamma_0}^\varepsilon,f_{\varGamma_0}^\varepsilon\big)_{H(\varGamma_0)}
= \big\|\Pi_\varGamma\circ K_{\varGamma_0}^\varGamma\big(g_{\varGamma_0}-\gamma_{\varGamma_0}h^\varepsilon\big)\big\|_{H(\varGamma)}^2.
\end{equation}
On the other hand, observing that $\Pi_\varGamma\big(K_{\varGamma_0}^\varGamma g_{\varGamma_0}-g^h_\varGamma\big) =K_{\varGamma_0}^\varGamma g_{\varGamma_0}-g^h_\varGamma$, we get:
\begin{multline*}
\left|\big\|\Pi_\varGamma\circ K_{\varGamma_0}^\varGamma\big(g_{\varGamma_0}-\gamma_{\varGamma_0}h^\varepsilon\big)\big\|^2_{H(\varGamma)}
-\big\| K_{\varGamma_0}^\varGamma g_{\varGamma_0}-g^h_\varGamma \big\|^2_{H(\varGamma)}\right|
\leqslant \\
\big\|\Pi_\varGamma \big(g^h_\varGamma-K_{\varGamma_0}^\varGamma\circ\gamma_{\varGamma_0}h^\varepsilon\big)
\big\|_{H(\varGamma)}\Big[\big\|\Pi_\varGamma \big(g^h_\varGamma-K_{\varGamma_0}^\varGamma\circ\gamma_{\varGamma_0}h^\varepsilon\big)
\big\|_{H(\varGamma)}+2\big\| K_{\varGamma_0}^\varGamma g_{\varGamma_0}-g^h_\varGamma\big\|_{H(\varGamma)}
\Big],
\end{multline*}
which, with the estimate \eqref{eq:est} gives for every $\varepsilon<1$:
\begin{equation}
\left|\big\|\Pi_\varGamma\circ K_{\varGamma_0}^\varGamma\big(g_{\varGamma_0}-\gamma_{\varGamma_0}h^\varepsilon\big)\big\|^2_{H(\varGamma)}
-\big\| K_{\varGamma_0}^\varGamma g_{\varGamma_0}-g^h_\varGamma \big\|^2_{H(\varGamma)}\right|<\varepsilon.
\end{equation}
By construction:
$$\big(K_{\varGamma_0}^\varGamma g_{\varGamma_0}-g^h_\varGamma,K_{\varGamma_0}^\varGamma \circ\gamma_{\varGamma_0}h
\big)_{H(\varGamma)}=0,$$
but also, according to \eqref{PRO:41}, for every $h\in\mathscr H(\varOmega)$:
$$\big(K_{\varGamma_0}^\varGamma g_{\varGamma_0}-g^h_\varGamma,K_{\varGamma_0}^\varGamma \circ\gamma_{\varGamma_0}h
\big)_{H(\varGamma)}=\big(K_{\varGamma}^{\varGamma_0}\big(K_{\varGamma_0}^\varGamma g_{\varGamma_0}-g^h_\varGamma\big)
,\gamma_{\varGamma_0}h\big)_{H(\varGamma_0)}.$$
This implies, with Lemma~\ref{LEM:51} that $\mathscr S_{\varGamma}\big(K_{\varGamma_0}^\varGamma g_{\varGamma_0}-g^h_\varGamma\big)$ is harmonic outside $\mathcal O$. On the other hand, inside $\mathcal O$, 
$\mathscr S_{\varGamma}\circ K_{\varGamma_0}^\varGamma g_{\varGamma_0}=
\mathscr S_{\varGamma_0}g_{\varGamma_0}$ and $\mathscr S_{\varGamma}g^h_\varGamma$ is harmonic (because $g^h_\varGamma$ 
is in the space $\overline{K_{\varGamma_0}^\varGamma\mathscr H(\varOmega)}$) and therefore $\Delta \mathscr S_{\varGamma}\big(K_{\varGamma_0}^\varGamma g_{\varGamma_0}-g_\varGamma\big)=h$. According to \eqref{eq:zero_fct}, we have:
\begin{equation}
\big\| K_{\varGamma_0}^\varGamma g_{\varGamma_0}-g^h_\varGamma \big\|^2_{H(\varGamma)}=\int_{\mathcal O}h^2\,{\rm d}m.
\end{equation}
\end{subequations}
Combining the equations \eqref{eq:concl} leads to the conclusion.
\end{proof}
Using the polarization identity, we prove:
\begin{cor}
\label{cor:mom}
For every $\varepsilon>0$ and every $k,m\in\mathbb N$, there exist $f_k^\varepsilon$ and $f_m^\varepsilon$ in the complex space
$\mathcal H(\varGamma_0)+i\mathcal H(\varGamma_0)$ 
such that:
$$\left|\big(K_\varGamma f_k^\varepsilon,f_m^\varepsilon\big)_{H(\varGamma_0)}- 
\int_{\mathcal O}z^k\overline{z^m}\,{\rm d}m(z)\right|<\varepsilon.$$
\end{cor}
We have now reached our goal: turning the reconstruction problem \ref{pb2} into a shape-from-moments problem. Indeed, Theorem~\ref{theo:fact} 
explains how to compute the operator $K_\varGamma$ from the measurement operator $\varLambda_\varGamma$ and Corollary~\ref{cor:mom} 
shows that a suitable choice of inputs can be used to calculate with an arbitrary precision the complex moments of the obstacles.
%
 \section{Using a finite section of all the complex moments}
 \label{SEC:6}
For a positive integer $n$ we assume known the complex moments of $\mathcal O$:
 $$\int_{\mathcal O}\overline{z^k}z^\ell\,{\rm d}m(z)\forallt k\leqslant n,\,\ell\leqslant n.$$
Applying a Gram-Schmidt process to the family of monomials $\{1,z,z^2,\ldots\,z^n\}$ we can compute from these moments the so-called Bergman polynomials $P_0$, $P_1,\ldots,P_n$, i.e. the polynomials orthonormalized for the scalar product of $L^2(\mathcal O)$ and such that $P_k$ is of degree 
$k$ for every $k=0,\ldots,n$. Following the ideas developed in  \cite{Gustafsson:2009aa}, we introduce the function $\varTheta_n$ defined by:
\begin{equation}
\label{def_Thea}
\varTheta_n(z)=\frac{1}{\sqrt{\pi\sum_{j=0}^n |P_j(z)|^2}}\forallt z\in\mathbb C.
\end{equation}
When all the curves $\varGamma_j$ ($j=1,\ldots,n$) are analytic, this function is shown  to approximate the distance to $\varGamma$ in $\mathcal O$ while it
decays to zero at certain rates, as $n$ goes to $+\infty$ on $\varGamma$ and in $\mathbb C\setminus\overline{\mathcal O}$.
More precisely, ${\rm dist}(z,\varGamma)\leqslant \varTheta_n(z)$ for $z\in\mathcal O$ and there exists two positive constants 
$C_1$ and $C_2$ such that $C_1/n\leqslant \varTheta_n(z)\leqslant C_2/n$ for $z\in\varGamma$. 
We will see in Section~\ref{SEC:num} that plotting the level sets $\{z\in\mathbb C\,:\,\varTheta_n(z)=\lambda/n\}$ for certain values of
 $\lambda$ 
gives a pretty good 
approximation of the boundaries of the obstacles.
%
 \section{Using a finite section of the harmonic moments only}
 \label{SEC:7}
For a positive integer $n$ we assume known the harmonic moments:
$$\tau_\ell= \int_{\mathcal O} z^\ell\,{\rm d}m(z)\forallt \ell=0,\ldots,2n-1.$$
An algorithm of reconstruction of $\mathcal O$ from the $\tau_\ell $ is detailed in \cite{Gerber-Roth:2023aa}. The key idea is to 
approximate $\mathcal O$ by  so-called quadrature domains. The method consists of two stages:
\begin{enumerate}
\item Determine complex weights $c_j$ and complex nodes $z_j$ ($j=1,\ldots,n$) such that
\begin{subequations}
\begin{equation}
\label{eq:st1}
\sum_{j=0}^n c_j z_j^\ell=\tau_j\forallt \ell=0,\ldots,2n-1.
\end{equation}
Such a system of equations is called a Prony's system.
\item Determine a domain $\mathcal O_n$ satisfying the quadrature identity:
\begin{equation}
\label{quad_dom}
\int_{\mathcal O_n} z^\ell\,{\rm d}m(z)=\sum_{j=0}^n c_j z_j^\ell\forallt \ell\in \mathbb N.
\end{equation}
\end{subequations}
By construction, $\mathcal O$ satisfies the identities above for $\ell=0,\ldots,2n-1$ but we emphasize that $\mathcal O_n$ is required 
to satisfy these identities for all the integers $\ell$. Such a domain is called a quadrature domain.
\end{enumerate}
Regarding the first stage, we introduce the polynomial:
\begin{equation*}
    P_n(z) = \begin{vmatrix}
    \tau_0     & \tau_1 & \cdots & \tau_{n-1}  & \tau_n\\
    \tau_1     & \tau_2 & \cdots & \tau_n      & \tau_{n+1}\\
    \vdots     & \vdots &        & \vdots      & \vdots\\
    \tau_{n-1} & \tau_n & \cdots & \tau_{2n-2} & \tau_{2n-1}\\
    1          & z      & \cdots & z^{n-1}     & z^n
    \end{vmatrix},
\end{equation*}
which enters the statement of the following result proved in \cite{Gerber-Roth:2023aa}:
\begin{theorem}
    \label{existence_prony}
    Equations \eqref{eq:st1} admit a solution if and only if the polynomial $P_n$ admits $n$ simple roots. 
    In this case, this solution is unique and the nodes $z_{1}, \ldots, z_{n}$ are the roots of $P_n$.
\end{theorem}
The explicit determination of the nodes $z_j$ and weights $c_j$ ($j=1,\ldots,n$) is carried out by means of 
the matrix pencil method  described in \cite{Golub:2000aa}. Introducing the Hankel matrices:
\begin{subequations}
\begin{equation}
\label{eq:Hank}
    \mathbb{H}_0^{(n)} = \begin{pmatrix}
        \tau_0     & \tau_1   & \dots  & \tau_{n-1} \\
        \tau_1     & \tau_2   & \dots  & \tau_{n} \\
        \vdots     & \vdots   & \ddots & \vdots \\
        \tau_{n-1} & \tau_{n} & \dots  & \tau_{2n-2}
    \end{pmatrix}
\qquad\text{and}\qquad 
       \mathbb{H}_1^{(n)} = \begin{pmatrix}
        \tau_1   & \tau_2     & \dots  & \tau_{n} \\
        \tau_2   & \tau_3     & \dots  & \tau_{n+1} \\
        \vdots   & \vdots     & \ddots & \vdots \\
        \tau_{n} & \tau_{n+1} & \dots  & \tau_{2n-1}
    \end{pmatrix},
\end{equation}
the nodes $z_j$ are the solutions of the following generalized
eigenvalue problem:
\begin{equation}
\label{eq:eig}
    \exists\, \xi\in\mathbb C^n \setminus \left\{ 0  \right\}, \quad \mathbb{H}_0^{(n)} \xi = z\, \mathbb{H}_1^{(n)} \xi.
\end{equation}
If the $z_j$ are pairwise distinct, the weights $c_j$ are obtained by solving the Vandermonde linear system:
\begin{equation}
\label{eq:weight}
    \begin{pmatrix}
        1           & 1           & \dots  & 1 \\[-1mm]
        z_1         & z_2         & \dots  & z_{n} \\
        z_1^2     & z_2^2     & \dots  & z_n^2 \\
        \vdots      & \vdots      & \ddots & \vdots \\
        z_1^{n-1} & z_2^{n-1} & \dots  & z_n^{n-1}
    \end{pmatrix}
    \begin{pmatrix}
    c_1\\
    c_2\\
    c_3\\
    \vdots\\
    c_n
    \end{pmatrix}=\begin{pmatrix}
        \tau_0 \\ 
         \tau_1\\ 
          \tau_2 \\ 
        \vdots \\
        \tau_{n-1}
    \end{pmatrix}.
\end{equation}
\end{subequations}
One can verify that $z_1,\ldots,z_n$ and $c_1,\ldots,c_n$ solve \eqref{eq:st1}
(see \cite[Section 3]{Golub:2000aa}).  Since the inverse problem we are dealing with is ill-posed, it is not surprising that it leads to a
numerical method requiring the solution of two ill-conditioned problems: a generalized eigenvalue problem
with Hankel matrices  and a Vandermonde system. 
It is worth mentioning also that numerical instabilities increase  with $n$.
\par
%
We now turn our attention to the second step of the reconstruction method.
Disks are the simplest examples of quadrature domains, and can be shown to be the only ones for which $n=1$. 
 Rather counter-intuitively, there are many quadrature domains.  Actually,
every domain whose boundary consists in non-intersecting $\mathcal C^\infty$ Jordan curves is
arbitrarily close to a quadrature domain. 

Given a  set of nodes and weights, existence of a quadrature domain satisfying \eqref{quad_dom} is asserted  
in \cite[Theorem 2.4, (vi)]{Gustafsson:1990ab}, providing that the weights are real and positive (uniqueness does not hold in general). 
The
construction of this domain is strongly related to partial balayage of measures
(see \cite{Gustafsson:1994aa}) and free boundary problems (see \cite{Gustafsson:1996aa}).
In summary, if the weights $c_j$ ($j=1,\ldots,n$) are positive in \eqref{quad_dom}, a quadrature domain $\mathcal O_n$ satisfying 
\eqref{quad_dom} can 
be  obtained by means of the following identity:
\begin{subequations}
\label{eq:indicatrice}
\begin{equation}
 \mathbf{1}_{\mathcal O_n}=-\Delta V^n,
\end{equation}
where $V^n$ is the unique function achieving:
\begin{equation}
\min_{v\in K^n}\frac{1}{2}\int_\varOmega |\nabla v|^2\,{\rm d}m-\int_\varOmega v\,{\rm d}m,
\end{equation}
with:
\begin{equation}
U^n=-\frac{1}{2\pi}\sum_{j=1}^n c_j\ln|\cdot-z_j|\qquad\text{and}\qquad
K^n = \left\{ v \in H^1(\varOmega) \, : \, \gamma_{\varGamma_0}^dv = \gamma_{\varGamma_0}^dU^{n}  \text{ and }  
        v \leqslant U^{n} \text{ in } \varOmega \right\}.
\end{equation}
\end{subequations}
The function that associates to any set of nodes and weights the corresponding quadrature domain (by means of the steps \eqref{eq:indicatrice})
is called the partial balayage operator.

\section{Algorithm and numerical tests}
\label{SEC:num}
In this section we describe an algorithm derived from the results of Sections~\ref{SEC:6} and \ref{SEC:7} and we provide some numerical tests to illustrate the efficiency of the reconstruction methods. 
\par
\subsubsection*{Generation of the measurements}
Fix a positive integer $m\geqslant 3$, denote $m'=2m-1$ and for $k=1,\ldots,m'$   define the functions:
$$F_k(z)=\begin{cases} z^k&\text{if }k=1,\ldots,m\\
\overline{z}z^{k-m}/(4(k-m))&\text{if }k=m+1,\ldots,m'
\end{cases}\forallt z\in\mathbb C.$$
Notice that $\Delta \overline{z}z^j/(4j)=z^{j-1}$ for every $j\geqslant 1$ and hence $\Delta^2 F_k=0$ for every $k=1,\ldots,m'$.
 Our first objective is to generate measurements corresponding to the boundary data 
$f_k=\widetilde \gamma_{\varGamma_0} F_k$ where $\widetilde \gamma_{\varGamma_0}=
\Pi_{\varGamma_0}\circ  \gamma_{\varGamma_0}$.

Computations are based on a BEM involving   biharmonic single-layer potentials. For every $k=1,\ldots,m'$, we compute the complex densities $q_j^k$ ($j=0,\ldots,n$) that solve the system 
of integral equations:
\begin{subequations}
\begin{alignat}{3}
\gamma_{\varGamma_0}\left[\mathscr S_{\varGamma_0}q^k_0 +\mathscr S_{\varGamma_1}q^k_1 +\ldots+\mathscr S_{\varGamma_N}q^k_N \right]&=\gamma_{\varGamma_0} F_k&\quad&\text{ on }\varGamma_0\\
\gamma_{\varGamma_j}\left[\mathscr S_{\varGamma_0}q^k_0 +\mathscr S_{\varGamma_1}q^k_1 +\ldots+\mathscr S_{\varGamma_N}q^k_N \right]&=0&&\text{ on }\varGamma_j
\quad(j=1,\ldots,N).
\end{alignat}
\end{subequations}
The function $\psi_k=\sum_{j=0}^N\mathscr S_{\varGamma_j} q_j^k$ is not the solution we are looking for as it satisfies neither  
the conditions
$\gamma_{\varGamma_0}\psi=f_k$ nor the boundary condition \eqref{sys_courant_1_eq3} (such that \eqref{sys_courant_1_eq6} holds).
So, we compute for every $j,k=0,\ldots,N$, the densities $p_j^k$ such that:
$$
\gamma_{\varGamma_j}\left[\mathscr S_{\varGamma_0}p^k_0 +\mathscr S_{\varGamma_1}p^k_1 +\ldots+\mathscr S_{\varGamma_N}p^k_N \right]=(\delta_{jk},0)\quad\text{ on }\varGamma_j
\quad(j=0,1,\ldots,N).
$$
It follows that $\xi_k=\sum_{j=0}^N\mathscr S_{\varGamma_j}p_j^k$ is the function introduced in 
Section~\ref{SEC:2} and that solves System \eqref{defxi}.
Then, we determine the constants $\alpha_{j,k}$ ($j=0,\ldots,N$, $k=1,\ldots,m'$) by inverting the linear system:
$$\big[\big\langle \widehat{\mathbf 1}_{\varGamma_j},q_j^k\big\rangle\big]_{0\leqslant j\leqslant N\atop 1\leqslant k\leqslant m'}
=\big[\big\langle \widehat{\mathbf 1}_{\varGamma_j},p_j^k\big\rangle\big]_{0\leqslant j\leqslant N\atop 0\leqslant k\leqslant N}
\big[\alpha_{j,k}\big]_{0\leqslant j\leqslant N\atop 1\leqslant k\leqslant m'},$$
and we define $\tilde q_j^k=q_j^k-\sum_{i=0}^N\alpha_{i,k}p^i_j$. This time,
the function $\widetilde\psi_k=\sum_{j=0}^N\mathscr S_{\varGamma_j}\tilde q_j^k$ satisfies all the conditions mentioned above.
Recall that we assume we have the operator $\varLambda_\varGamma$ (defined in \eqref{fct_Lambda}) at our disposal. Since 
the definition of the operator $\varLambda_0$ depends only on the boundary $\varGamma_0$, we have access, as already mentioned, to the operator $R_\varGamma=S_{\varGamma_0}\circ \big(\varLambda_\varGamma-\varLambda_0\big)$ (see Theorem~\ref{theo:fact}) and then also to the operator $V_\varGamma={\rm Id}+R_\varGamma$. We easily verify that:
$$V_\varGamma f_k=\tilde q_0^k\forallt k=1,\ldots,m',$$
and this will be the measurements we shall use for the reconstruction.
\subsubsection*{Computing the scalar product on $\varGamma$}
The objective of this subsection is to explain how to compute the scalar products:
$$\big( \widetilde \gamma_{\varGamma}\overline{F_j}, \widetilde \gamma_{\varGamma}F_k\big)_{H(\varGamma)}\forallt 
j,k=1,\ldots,m'.$$
According to Theorem~\ref{theo:fact},   $K_\varGamma={\rm Id}-V_\varGamma^{-1}$ and it follows that:
\begin{equation}
\label{pldeuj}
\big(K_\varGamma \overline{f_j},f_k\big)_{H(\varGamma_0)}=\big( \overline{f_j},f_k\big)_{H(\varGamma_0)}-\big(V_\varGamma^{-1} \overline{f_j},f_k\big)_{H(\varGamma_0)}\forallt j,k=1,\ldots,m'.
\end{equation}
Notice that, by definition 
$$\big( \overline{f_j},f_k\big)_{H(\varGamma_0)}=\big(\widetilde\gamma_{\varGamma_0} \overline{F_j},
\widetilde\gamma_{\varGamma_0}F_k\big)_{H(\varGamma_0)},$$
and according to Theorem~\ref{Theo:52}, for every $j,k=1,\ldots,m'$ we have:
$$
\big(K_\varGamma  \overline{f_j},f_k\big)_{H(\varGamma_0)}=\Big(  \Pi_\varGamma\circ K_{\varGamma_0}^\varGamma  \overline{f_j},
K_{\varGamma_0}^\varGamma f_k\Big)_{H(\varGamma)}=
\big( \widetilde \gamma_{\varGamma} \overline{F_j}, \widetilde \gamma_{\varGamma}F_k\big)_{H(\varGamma)}.$$
We define the positive-definite symmetric matrices: 
$$Q^{\varGamma_0}_m=\big[( \overline{f_j},f_k)_{H(\varGamma_0)}\big]_{1\leqslant j\leqslant m'\atop 1\leqslant k\leqslant m'}\qquad\text{and}\qquad
Q^{\varGamma}_m=\big[(K_\varGamma  \overline{f_j},f_k)_{H(\varGamma)}\big]_{1\leqslant j\leqslant m'\atop 1\leqslant k\leqslant m'},$$
and also:
$$V_m=\big[(V_\varGamma  \overline{f_j},f_k)_{H(\varGamma_0)}\big]_{1\leqslant j\leqslant m'\atop 1\leqslant k\leqslant m'}
\qquad\text{and}\qquad W_m=\big[(V^{-1}_\varGamma  \overline{f_j},f_k)_{H(\varGamma_0)}\big]_{1\leqslant j\leqslant m'\atop 1\leqslant k\leqslant m'}.$$
The matrix $W_m$ is approximated by the matrix $Q^{\varGamma_0}_mV_m^{-1}Q^{\varGamma_0}_m$ so that \eqref{pldeuj} 
can be rewritten in matrix form:
$$
Q_m^\varGamma=Q_m^{\varGamma_0}-Q^{\varGamma_0}_mV_m^{-1}Q^{\varGamma_0}_m.
$$
\subsubsection*{Computing the complex moments of the obstacles}
According to Lemma~\ref{LEM:51}, if $p$ is in $\mathcal H(\varGamma)$ and satisfies $(p,K_{\varGamma_0}^\varGamma\circ\gamma_{\varGamma_0}h)_{H(\varGamma)}=0$ for every $h\in\mathscr H(\varOmega)$, then 
$\Delta\mathscr S^+_{\varGamma}\widehat p=0$ and in this case:
$$\|\mathscr S_{\varGamma}\widehat p\|_{H(\varGamma)}^2=\|\Delta \mathscr S_\varGamma\widehat p\|_{L^2(\mathcal O)}^2.$$
To put these principles into practice, let us denote  $G_k=F_{m+k}$ and recall that $\Delta G_k=z^{k-1}$ for $k=1,\ldots,m-1$. 
We decompose the matrix $Q_m^\varGamma$ into 4 sub-matrices:
$$Q_m^\varGamma=\begin{bmatrix} X_m&Y_m\\Y_m^\ast&Z_m\end{bmatrix},$$
where, by construction: 
$$X_m=\big[(\widetilde\gamma_{\varGamma}\overline{z^j},\widetilde\gamma_{\varGamma}{z^k})_{H(\varGamma)}\big]_{1\leqslant j\leqslant m\atop1\leqslant k\leqslant m},\quad 
Y_m=\big[(\widetilde\gamma_{\varGamma}\overline{z^j},\widetilde\gamma_{\varGamma}{G_k})_{H(\varGamma)}\big]_{1\leqslant j\leqslant m\atop1\leqslant k\leqslant m-1}\quad
\text{and}\quad
Z_m=\big[(\widetilde\gamma_{\varGamma}\overline{G_j},
\widetilde\gamma_{\varGamma}{G_k})_{H(\varGamma)}\big]_{1\leqslant j\leqslant m-1\atop1\leqslant k\leqslant m-1}.$$
It follows that the entries of the matrix $\mathcal M^\varGamma_m=Z_m-Y^\ast_m X_m^{-1}Y_m$ are 
$$\big(\Pi_m\circ\gamma_\varGamma \overline{G_j},
\Pi_m\circ\gamma_\varGamma G_k\big)_{H(\varGamma)}\forallt j,k=1,\ldots,m-1,$$
where $\Pi_m$ is the orthogonal projection in $H(\varGamma)$ onto the orthogonal of the 
subspace spanned by $\{1,z,z^2,\ldots,z^{m}\}$. As a conclusion, 
for $m$ large enough, the entries of the matrix $\mathcal M^\varGamma_m$ are such that:
$$\big(\Pi_m\circ\gamma_\varGamma \overline{G_j},
\Pi_m\circ\gamma_\varGamma G_k\big)_{H(\varGamma)}\simeq
 \int_{\mathcal O}\overline{z^{j-1}}z^{k-1}\,{\rm d}m(z)\forallt j,k=1,\ldots,m-1.$$
\subsubsection*{Reconstruction with a finite section of all the complex moments}
Let $L_m$ be the upper triangular matrix entering the Cholevsky factorization of $\mathcal M^\varGamma_m$ i.e. $\mathcal 
M^\varGamma_m=L_m^\ast L_m$.
Then, it is easy to verify that the columns of $L_m$ are the coefficients of the Bergman polynomials. It is then possible to 
evaluate the function $\varTheta_{m-2}$ (defined in \eqref{def_Thea}) at any point of $\varOmega$ and draw the set corresponding 
to $\varTheta_{m-2}= \lambda/(m-2)$ for some values of $\lambda$.
\subsubsection*{Reconstruction with a finite section of the harmonic moments only}
With this approach, we use only the first row of the matrix $\mathcal M^\varGamma_m$. We fix an integer $n\leqslant (m-1)/2$ and we apply the 
method described in Section~\ref{SEC:7}. As explained there, to solve the Prony's system \eqref{eq:st1} we construct the Hankel matrices 
\eqref{eq:Hank} and solve the generalized eigenvalue problem \eqref{eq:eig} which provides the values of the nodes $z_j$. 
The weights $c_j$ are obtained by solving the linear system \eqref{eq:weight}. We replace $c_j$ by $\real(c_j)$ when $\imag(c_j)$ is 
non-zero and we solve the convex minimization problem \eqref{eq:indicatrice} (partial balayage process) 
to obtain (hopefully) approximations of the obstacles.
\subsubsection*{Numerical tests}
For all the tests, the domain $\varOmega$ is a disk centered at the origin and of radius 1. 
\medskip
\par
\noindent\underline{{\it Example 1}}.
We  consider a single cross-shaped 
obstacle whose boundary is parameterized as:
$$\begin{pmatrix}x_1(\theta)\\x_2(\theta)\end{pmatrix}
=\left[0.25(1+0.4\cos(4\theta))\begin{pmatrix}
\cos(\theta) \\
\sin(\theta)
\end{pmatrix}  + 0.2
\begin{pmatrix}1\\1
\end{pmatrix}\right],\qquad(\theta\in[0,2\pi[).$$
Note that this is a non-trivial example (neither small, nor too close to the boundary $\varGamma_0$, nor a ball, nor even convex).
The reconstruction is performed with $m=13$ (corresponding to $2m-1=25$ measurements).  
It means that we have at our disposal the moments:
$$\int_{\mathcal O}\bar z^kz^n\,{\rm d}m(z)\forallt k,n=0,\ldots,11.$$
Numerical instabilities make it difficult to increase $m$ much further in this case (for larger $m$ the matrix $\mathcal M^\varGamma_m$ 
exhibits spurious negative eigenvalues). First we apply the method described in Section~\ref{SEC:6}.
Several level lines corresponding to $\varTheta_{11}=\lambda/(m-2)$ for $\lambda=0,2,0,4,0.6,0.8$ and $1$ are plotted on Fig.~\ref{fig_recont1}. The position of the obstacle is accurately found and the level lines outline clearly its convex hull. 
However, the details of the branches are not captured by the reconstruction. This would probably require $m$ to be larger.
\begin{figure}[h]
\centerline{
\begin{tabular}{cc}
\includegraphics[width=0.5\textwidth]{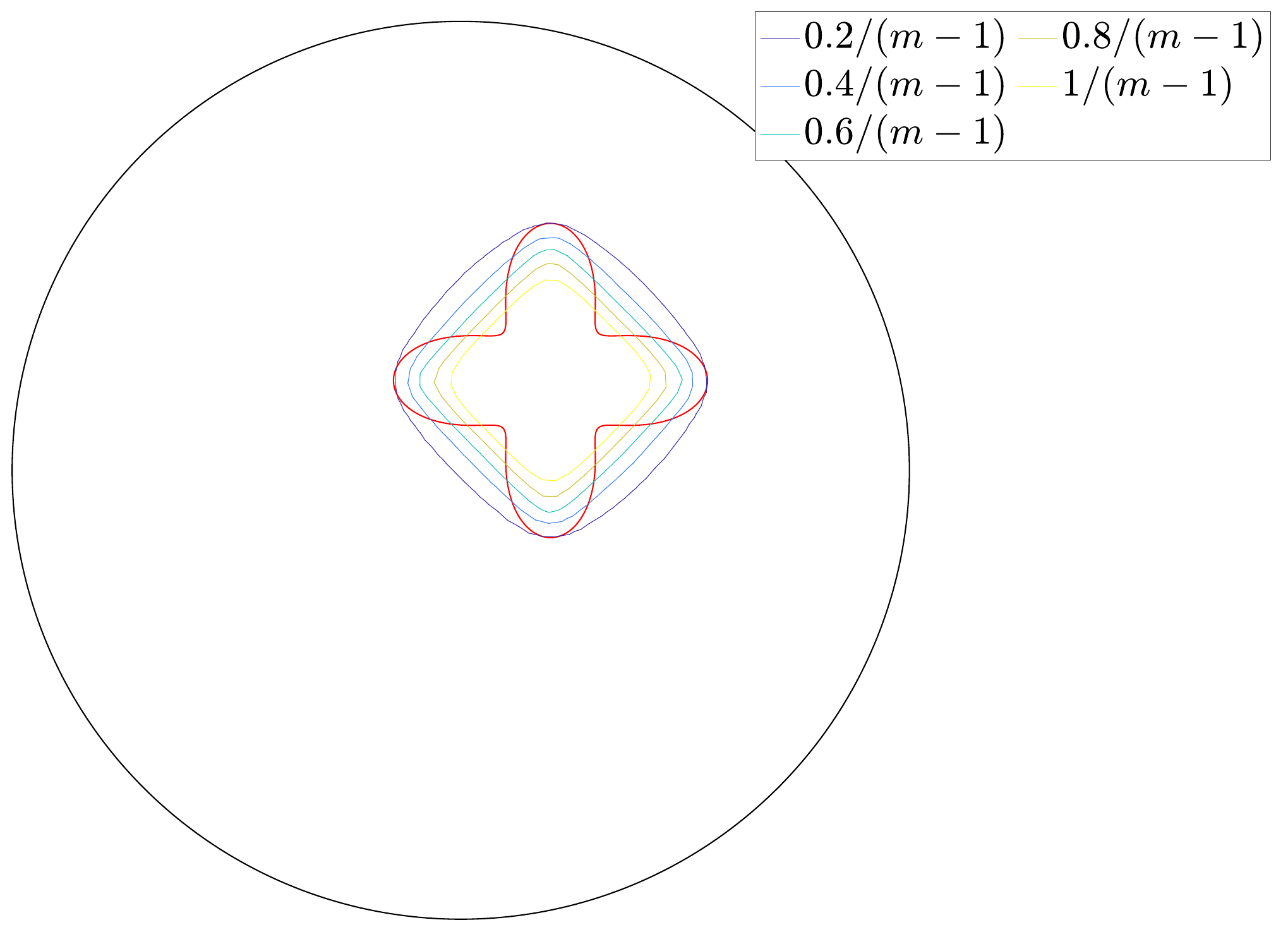}&\includegraphics[width=0.25\textwidth]{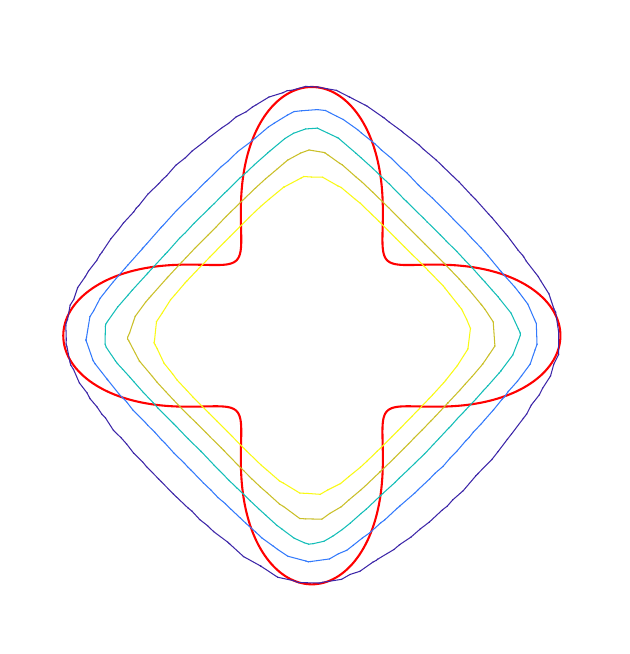}
\end{tabular}}
\caption{\label{fig_recont1}Reconstruction of a single cross-shaped obstacle using the method described in Section~\ref{SEC:6} (borrowed from \cite{Gustafsson:2009aa}) with 25 measurements (i.e. $m=13$). The image on the right is a close-up of the obstacle and its reconstruction.}
\end{figure}
\par
We apply now the method explained in Section~\ref{SEC:7} and for which only the harmonic moments:
$$\int_{\mathcal O} z^k\,{\rm d}m(z)\forallt k=0,\ldots,10,$$
are involved. We solve the related Prony's system \eqref{eq:st1} for $n=1,2,\ldots,5$ and for each case, we  represent in Fig.~\ref{fig_frtp2} 
the disks centered at the nodes 
$z_j$ and of radii $\sqrt{\real(c_j)/\pi}$. Indeed, when these disk are pairwise disjoint, the mean value property for harmonic functions 
ensures that their union is a quadrature domain that satisfies equality \eqref{quad_dom}. For $n=1,2,3$, only one weight is non zero and hence only one 
disk is non-degenerated (obviously no partial balayage step is needed). For $n=4$, we obtain four disjoint disks, so no partial balayage 
step is necessary either in this case. On the other hand, the disks overlap for $n=5$, so we solve the convex minimization 
problem \eqref{eq:indicatrice}. The resulting quadrature domain is represented on Fig.~\ref{fig_frtp3}. The reconstruction is much better than on 
Fig.~\ref{fig_recont1}, as the outline of the cross can be clearly seen. Such reconstruction accuracy is remarkable for this type of problem.
\begin{figure}[h]
\begin{tabular}{ccc}
\includegraphics[width=0.3\textwidth]{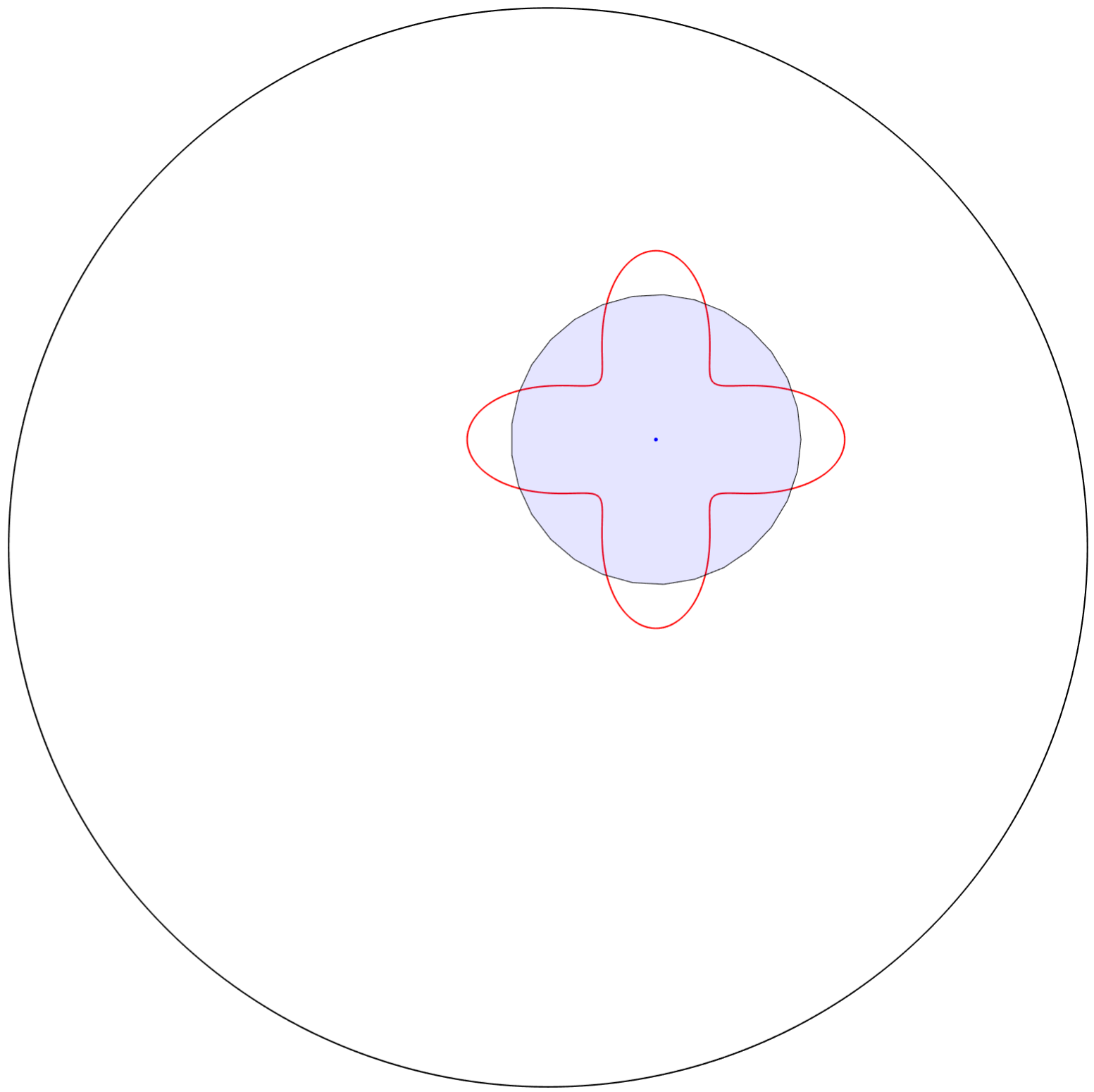}&\includegraphics[width=0.3\textwidth]{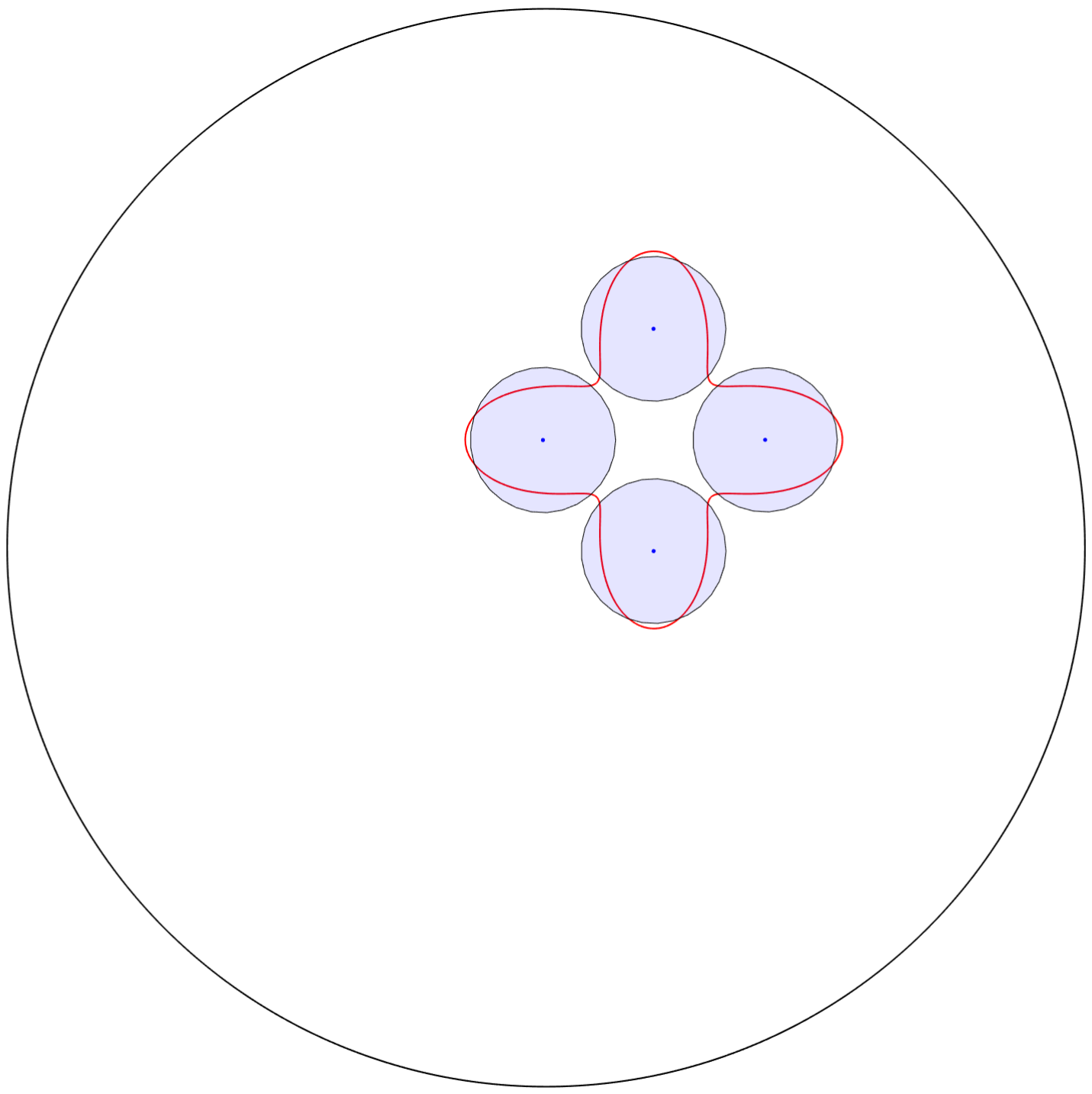}&
\includegraphics[width=0.3\textwidth]{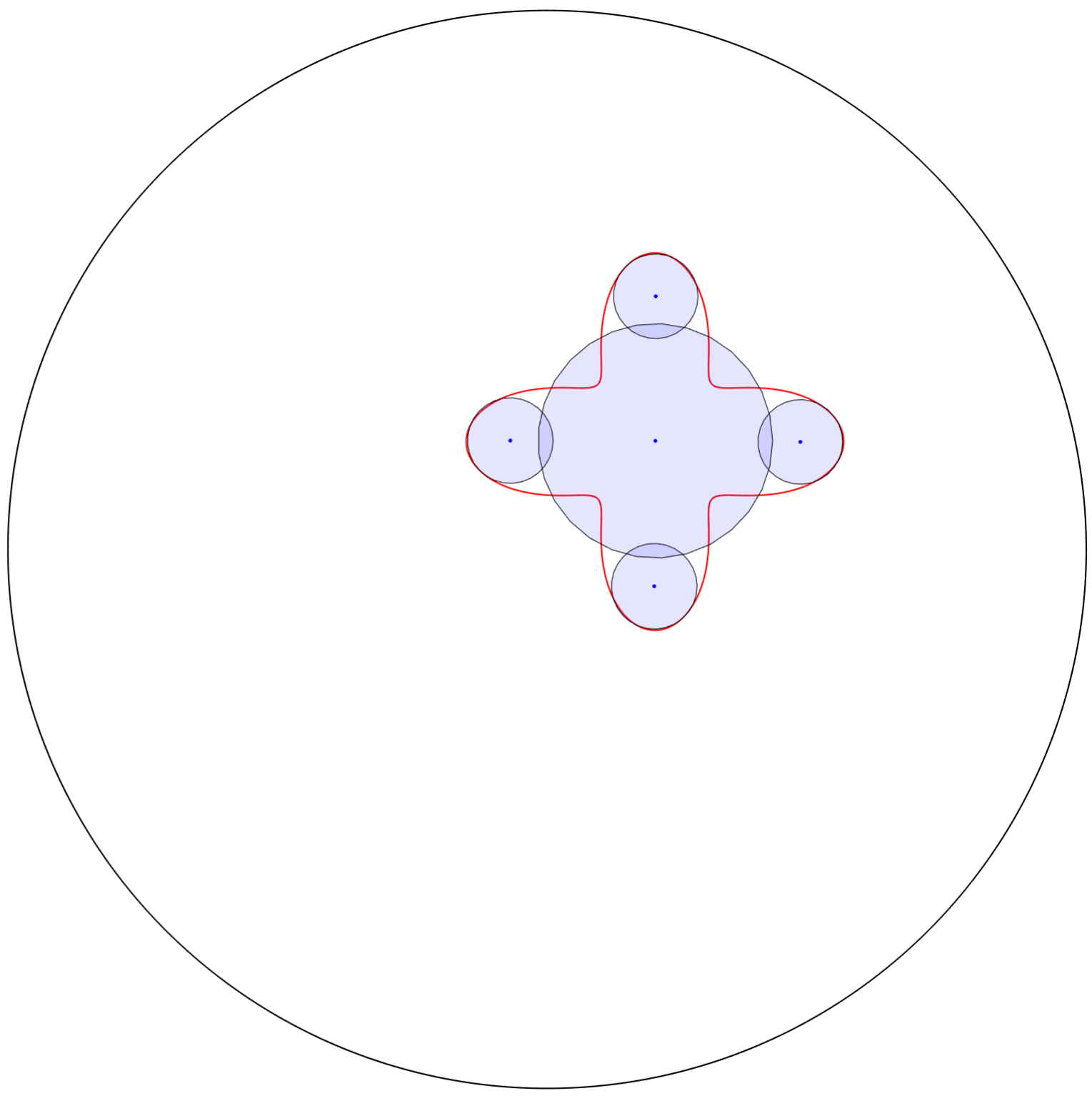}
\end{tabular}
\caption{\label{fig_frtp2}Reconstruction of a single obstacle with disks obtained by solving the Prony's system corresponding to the harmonic moments for $n=1,2,3$ (first picture), $n=4$ (second picture) and $n=5$ (third picture).}
\end{figure}
\begin{figure}[h]
\centerline{\begin{tabular}{ccc}
\includegraphics[width=0.3\textwidth]{cross_prony_5.pdf}&
\includegraphics[width=0.3\textwidth]{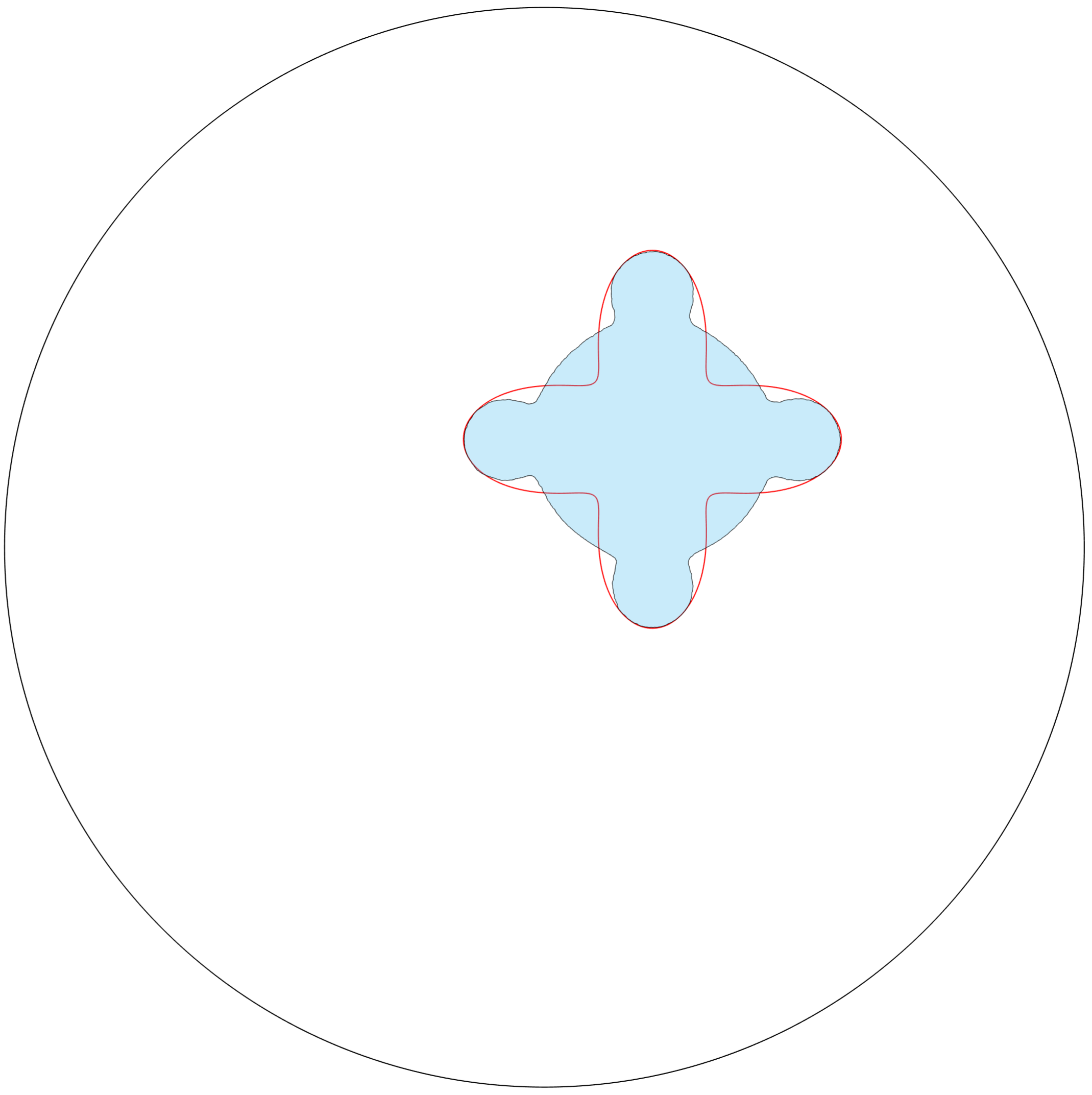}&
\includegraphics[width=0.31\textwidth]{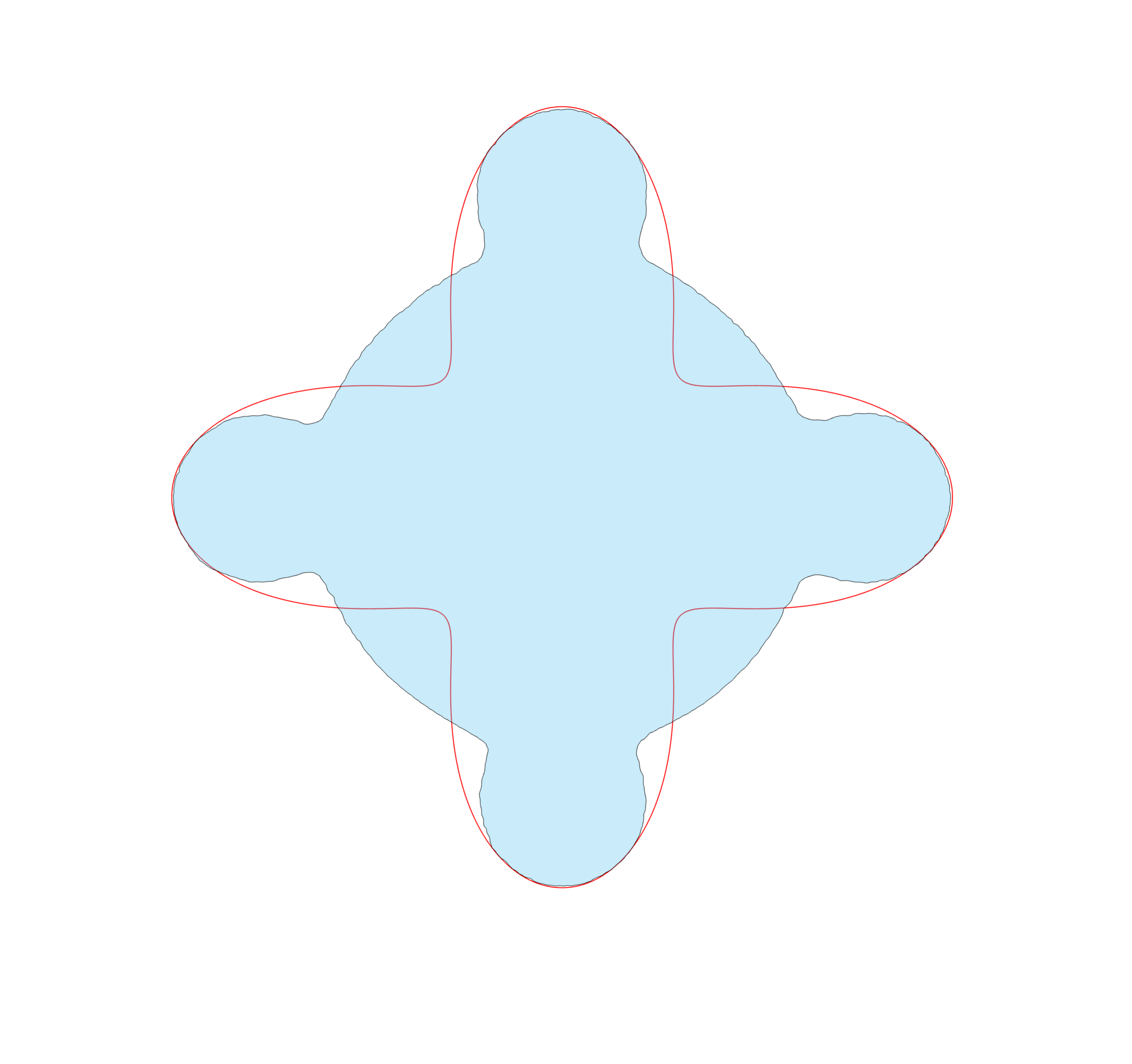}
\end{tabular}}
\caption{\label{fig_frtp3}Reconstruction of the cross-shaped obstacle after applying the partial balayage operator to the case $n=5$.}
\end{figure}
\medskip
\par
\noindent\underline{{\it Example 2}}.
We consider now two obstacles. 
The cross-shaped obstacle is retained but slightly translated and is now set as follows:
$$\begin{pmatrix}x_1(\theta)\\x_2(\theta)\end{pmatrix}
=\left[0.25(1+0.4\cos(4\theta))\begin{pmatrix}
\cos(\theta) \\
\sin(\theta)
\end{pmatrix}  + 0.35
\begin{pmatrix}1\\1
\end{pmatrix}\right],\qquad(\theta\in[0,2\pi[).$$
We add an ellipse-shaped obstacle parameterized as:
$$\begin{pmatrix}x_1(\theta)\\x_2(\theta)\end{pmatrix}=\left[\begin{pmatrix}
0.25\cos(\theta)\\
 0.1\sin(\theta)
 \end{pmatrix} - 0.45\begin{pmatrix}1\\1
\end{pmatrix}\right],\qquad(\theta\in[0,2\pi[).$$
In this case, we can increase the parameter $m$  to $m=19$ (corresponding to 37 measurements). Some level lines of the function $\varTheta_{17}$  (corresponding to values of order $1/17$) are displayed in Fig.~\ref{fig_recont2}. Obstacle positions and sizes are recovered satisfactorily. Reconstruction quality is better for the ellipse than for the crosse (convexity seems to  play a role for this method).
\begin{figure}[h]
\centerline{\includegraphics[width=0.5\textwidth]{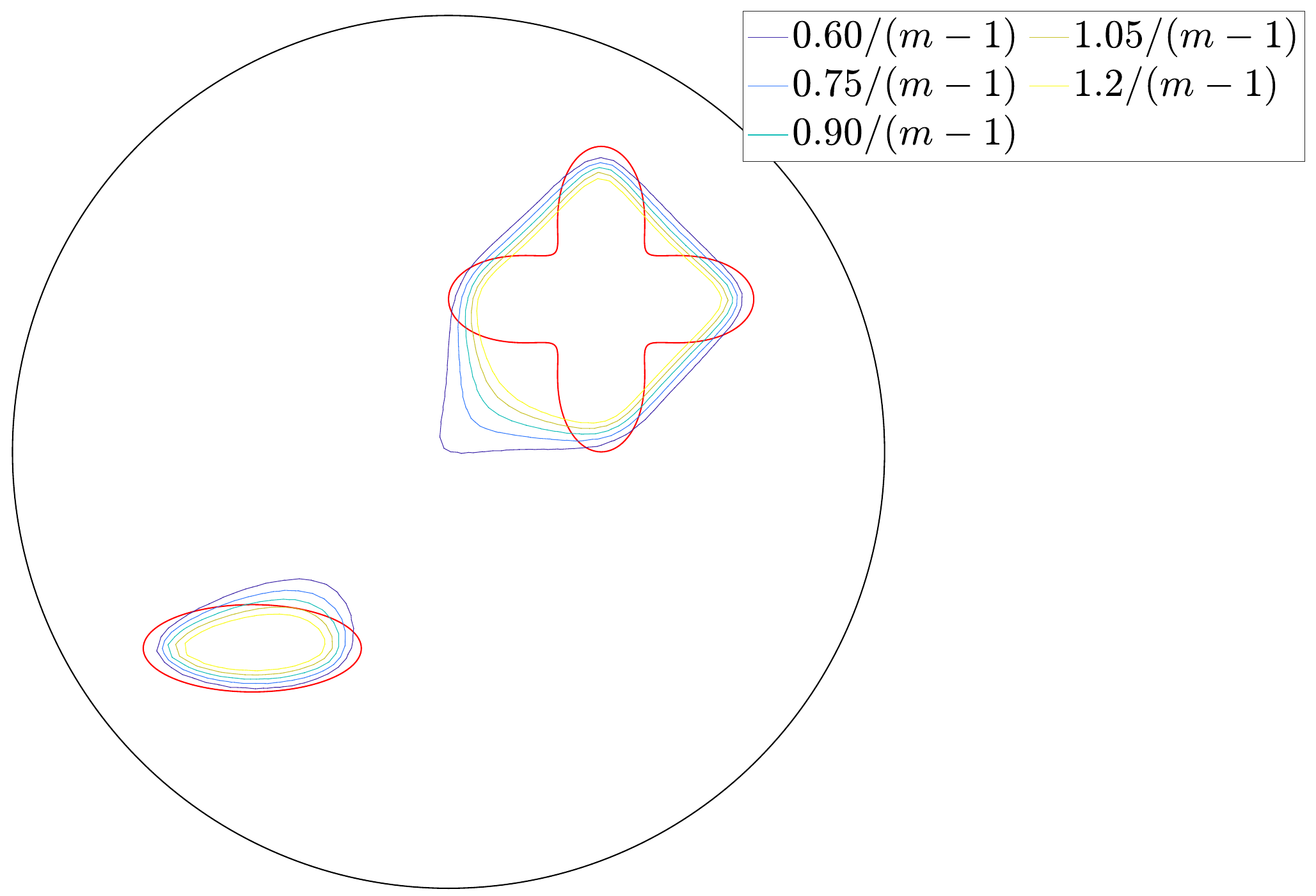}}
\caption{\label{fig_recont2}Reconstruction of two obstacles  using the method described in Section~\ref{SEC:6}. Some level 
lines of the function $\varTheta_{17}$ are represented.}
\end{figure}
\par
Using the first 18 harmonic moments only, we  apply now the method of Section~\ref{SEC:7} and plot the corresponding disks (as 
explained for the first example) for $n=2,5$ and $8$ in Fig~\ref{fig8}. For $n=2$ and $n=5$, no partial balayage step is needed as 
the disks do not overlap. For $n=8$, we apply the partial balayage operator by solving the convex minimization problem \eqref{eq:indicatrice}, 
which provides the quadrature domains represented in Fig.~\ref{fig_frtp4} (the two small spurious disks are neglected). 
Two branches of the cross can be identified and the  shape of the ellipse is correctly recovered. Once again, this is a very satisfactory reconstruction for this  demanding example.
\begin{figure}[h]
\begin{tabular}{ccc}
\includegraphics[width=0.3\textwidth]{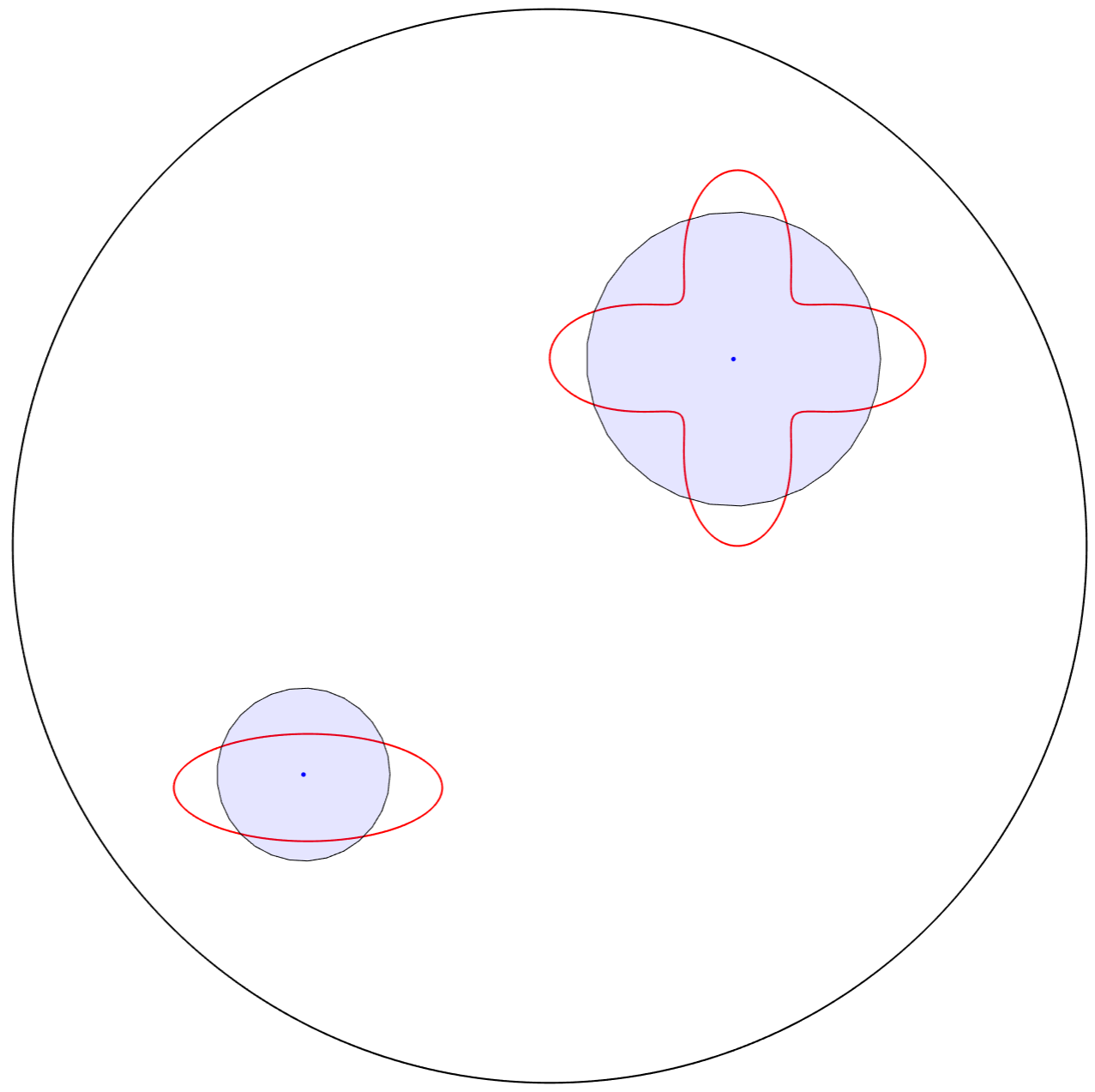}&\includegraphics[width=0.3\textwidth]{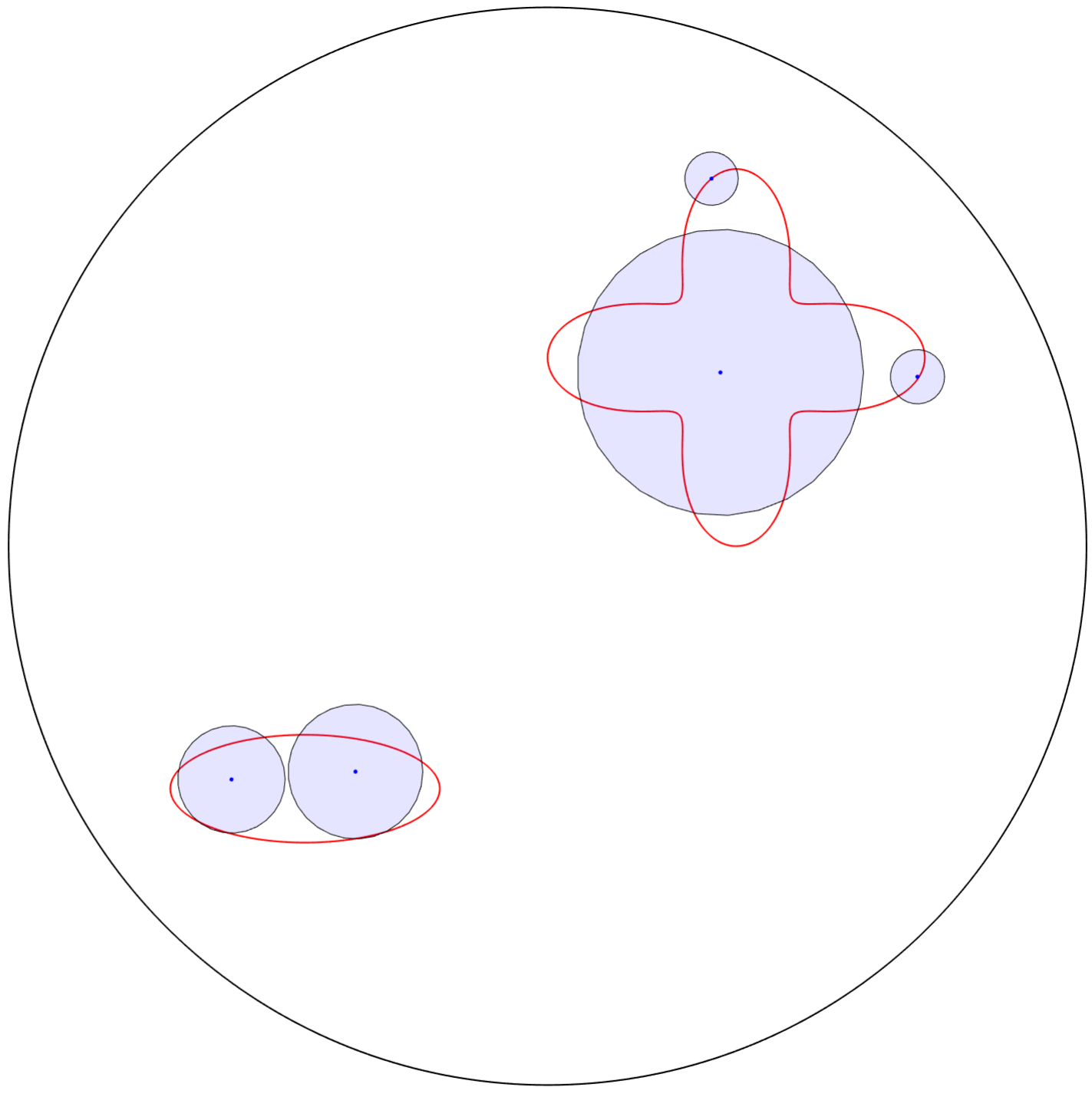}&
\includegraphics[width=0.3\textwidth]{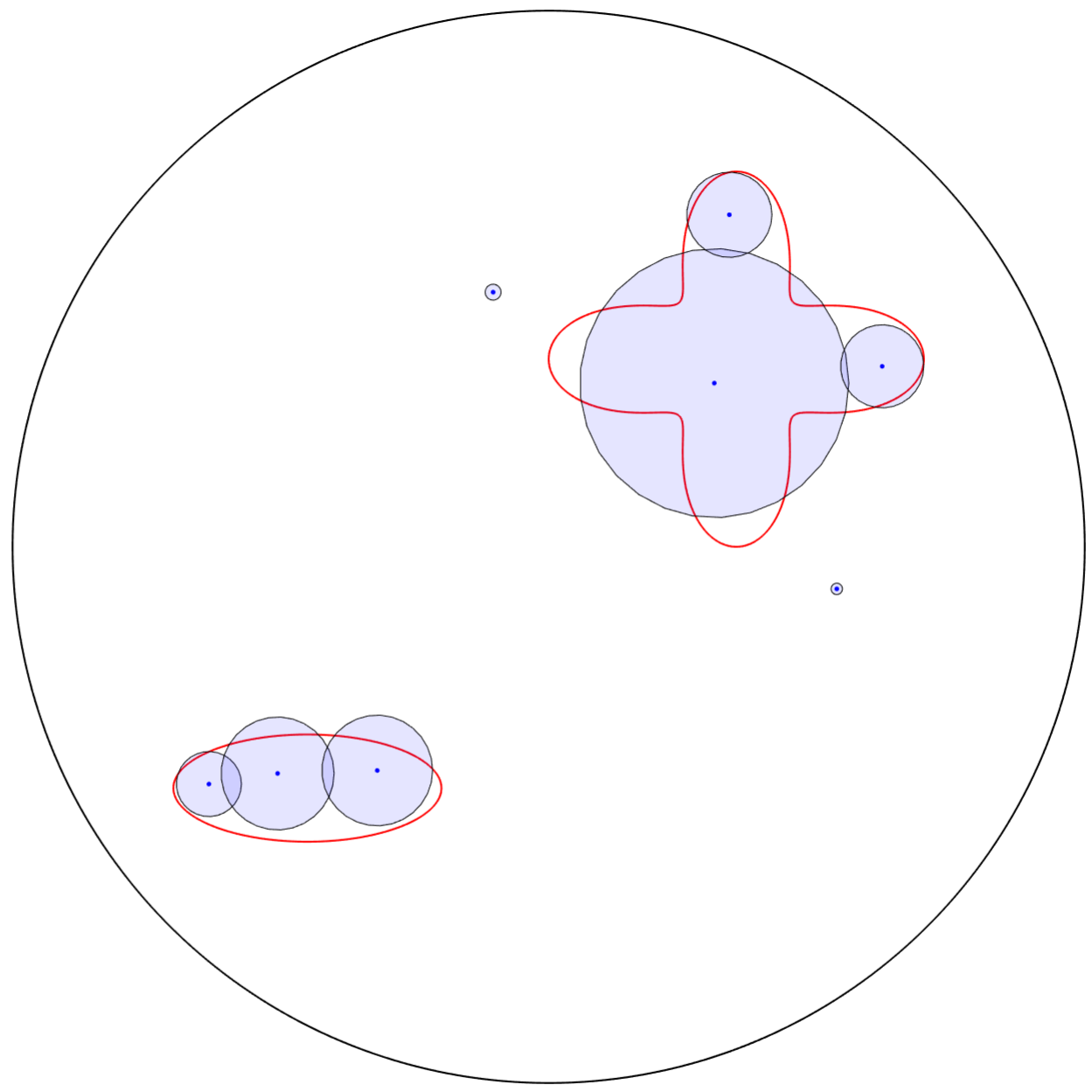}
\end{tabular}
\caption{\label{fig8}Reconstruction with 2, 5 and 8 disks. Two small spurious disks are observed in the case $n=8$. They are neglected 
for the partial balayage step.}
\end{figure}
\begin{figure}[h]
\begin{tabular}{ccc}
\includegraphics[width=0.3\textwidth]{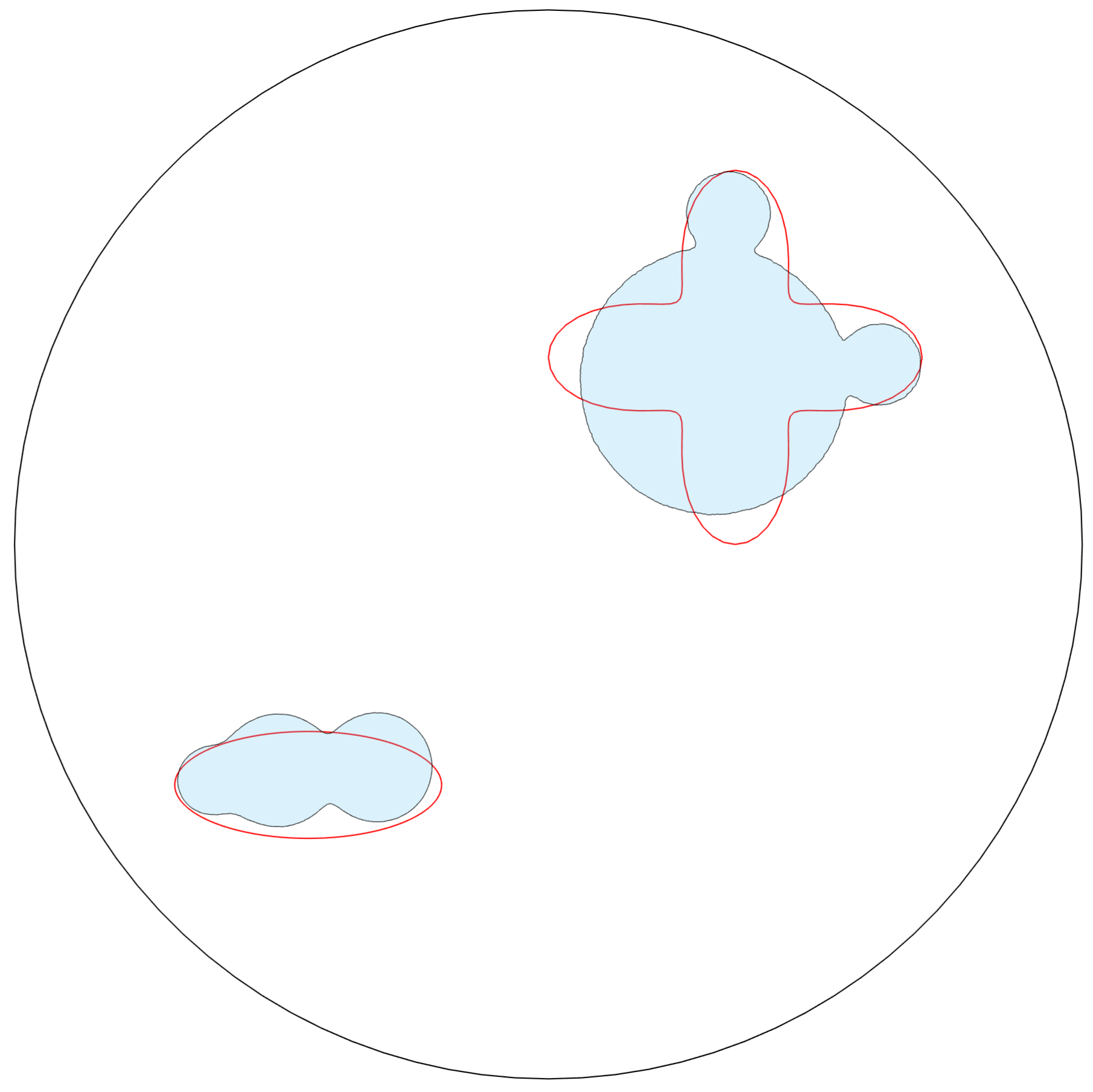}&\includegraphics[width=0.3\textwidth]{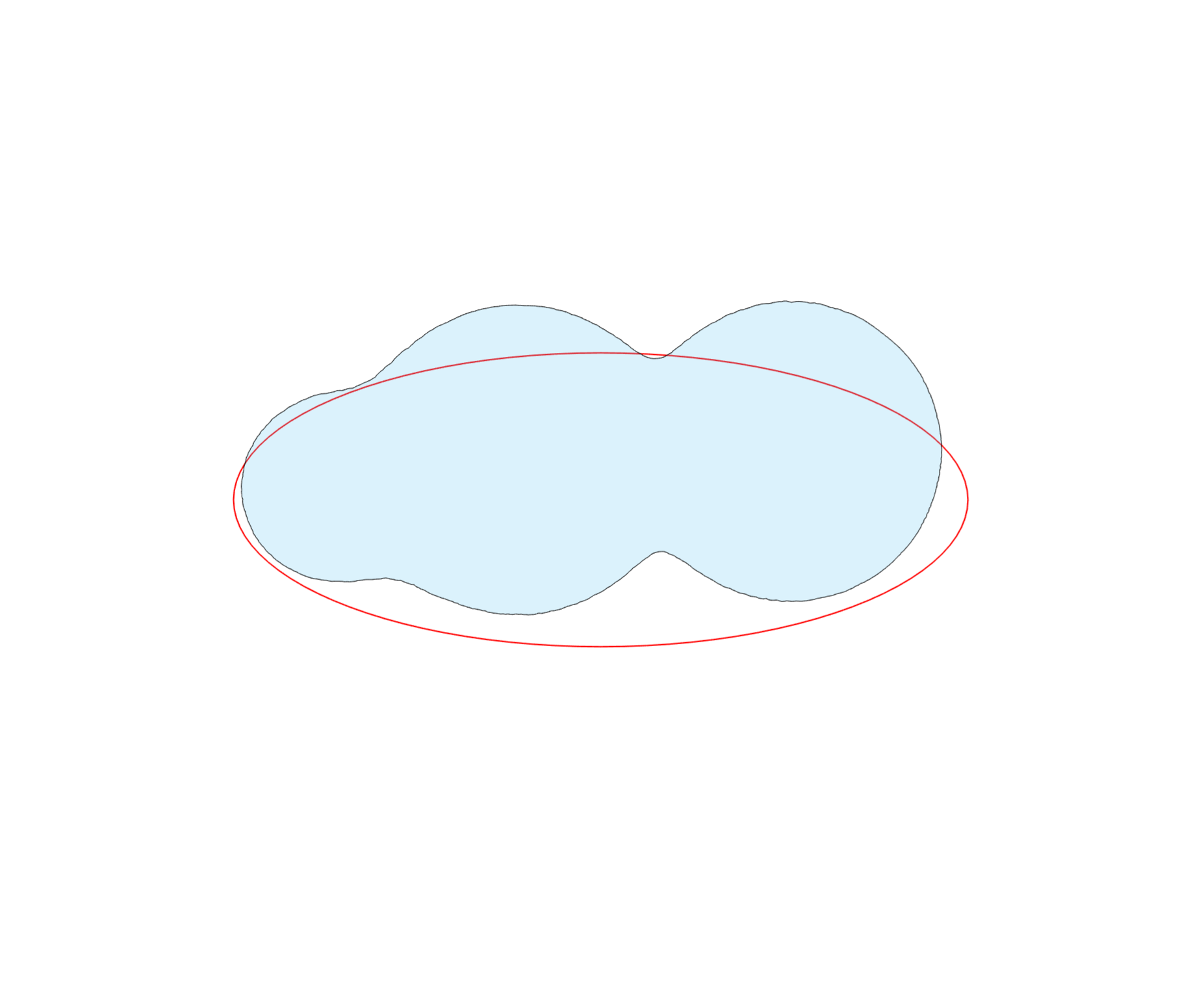}&
\includegraphics[width=0.3\textwidth]{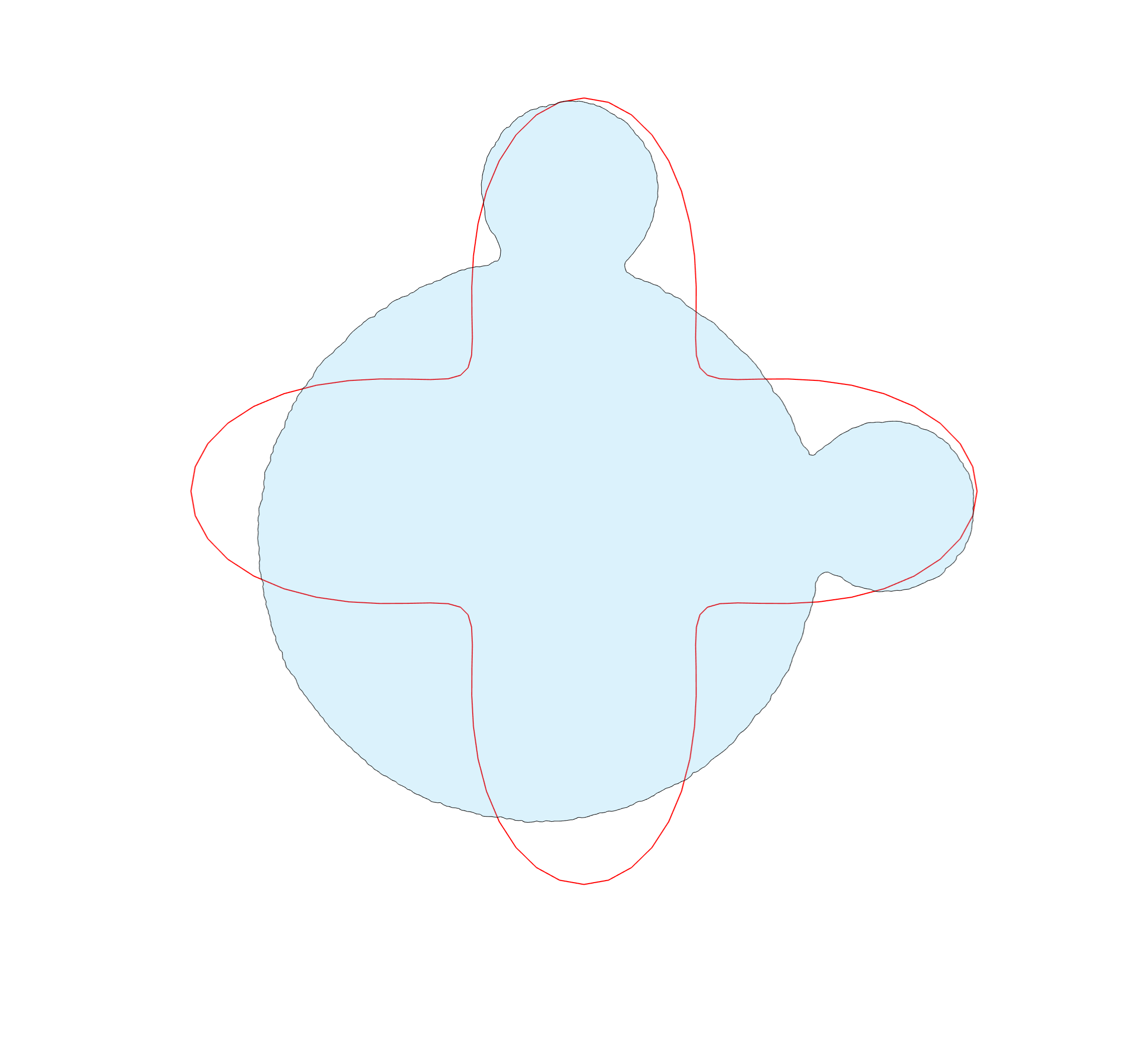}
\end{tabular}
\caption{\label{fig_frtp4}Quadrature domains obtained by applying the partial balayage operator to the case $n=8$ of Fig.~\ref{fig8}.}
\end{figure}
\medskip
\par
\noindent\underline{{\it Example 3}}.
The last case we consider is borrowed from \cite{Caubet:2016ab}. The obstacles are two squares centered at
$(-0.6, 0.3)$ and $(0.6, 0.3)$ with a distance between the center and the vertices equal to $0.2$. The interest of studying this case is twofold:  the obstacle boundaries are not smooth (unlike examples 1 and 2), and results can be compared with those provided 
in \cite{Caubet:2016ab} where the obstacles are a priori assumed to be star-shaped. The parameter $m$ is set to $m=18$ (35 measurements). The results obtained with the method of 
Section~\ref{SEC:6} 
are displayed in Fig.~\ref{fig_recont3} and those obtained with the method of Section~\ref{SEC:7} in Fig.~\ref{fig9} (no partial balayage 
step is applied). Both methods deliver highly satisfactory reconstructions.  %
\begin{figure}[h]
\centerline{
\begin{tabular}{cc}
\includegraphics[width=0.5\textwidth]{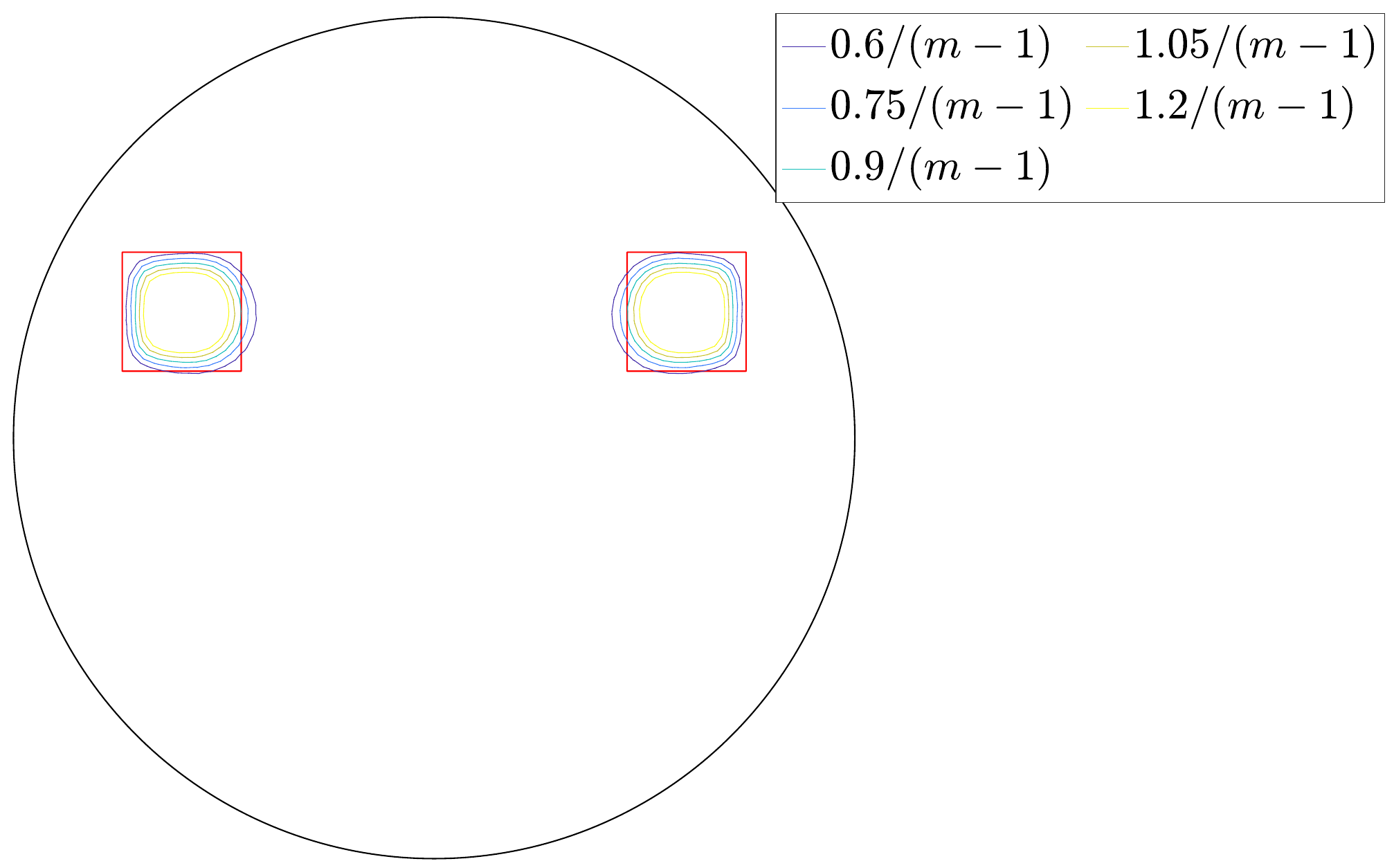}&
\includegraphics[width=0.3\textwidth]{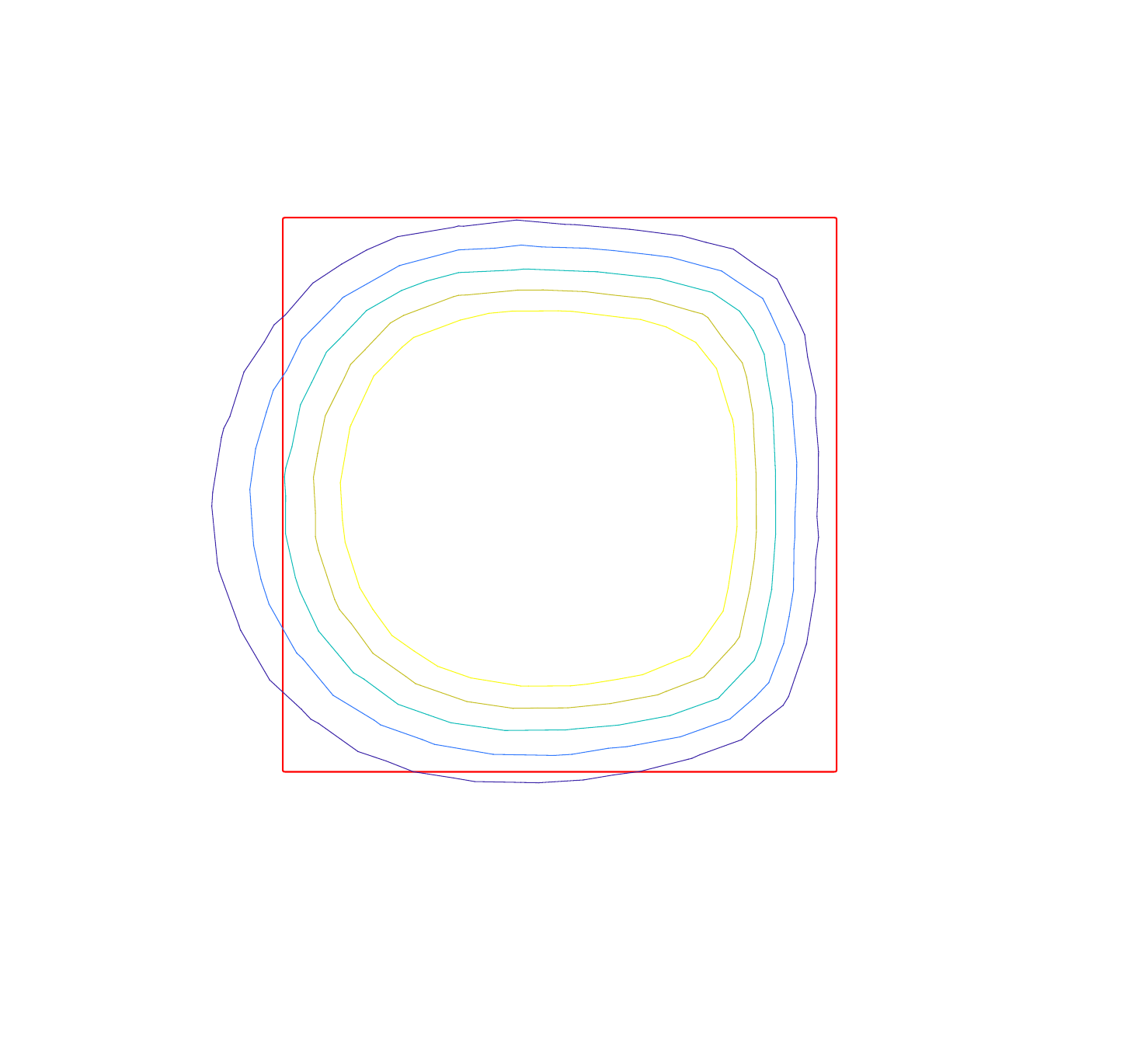}
\end{tabular}
}
\caption{\label{fig_recont3}Reconstruction of two squares  using the method described in Section~\ref{SEC:6}. Some level 
lines of the function $\varTheta_{16}$ are represented. The picture on the right is a close-up of the right square and its reconstruction.}
\end{figure}

\begin{figure}
\centerline{
\begin{tabular}{cc}
\includegraphics[width=0.3\textwidth]{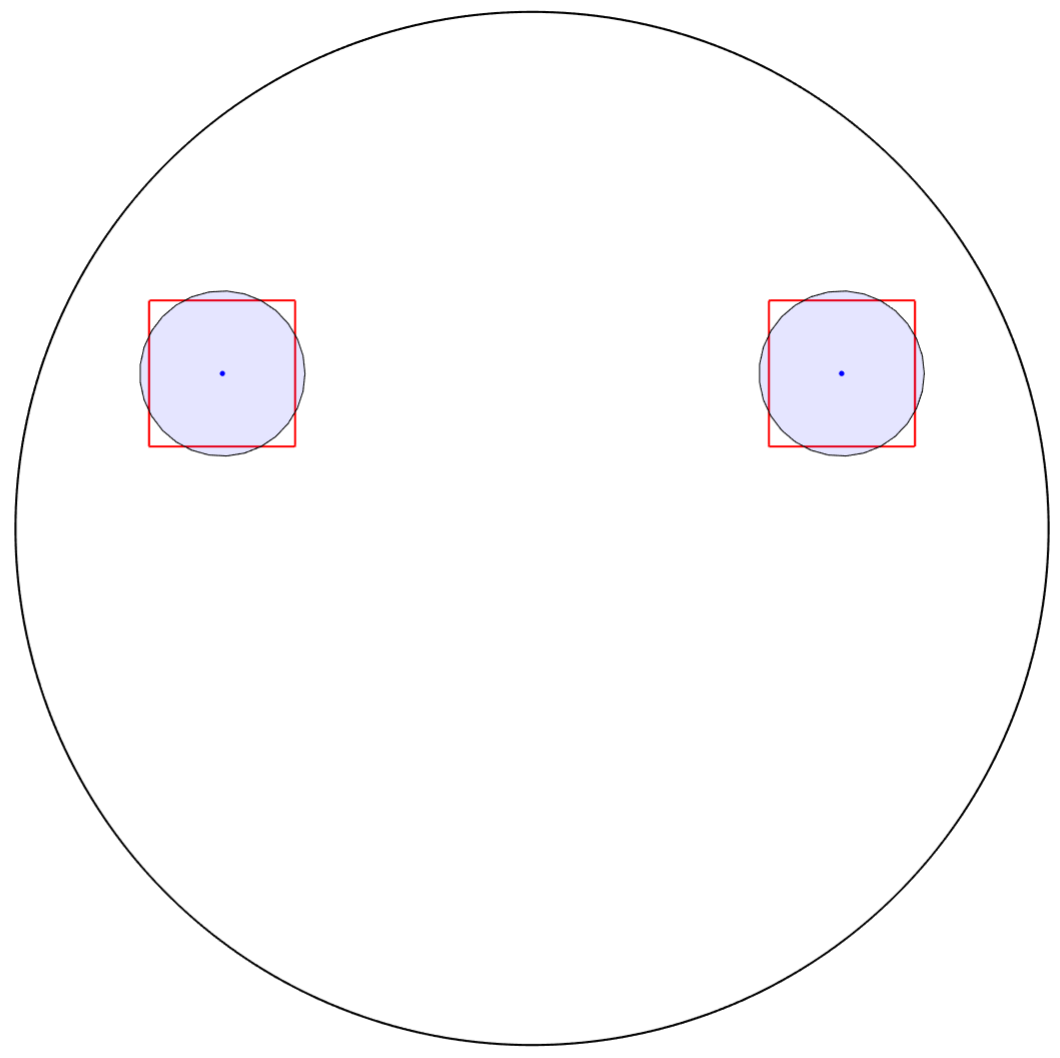}&\includegraphics[width=0.3\textwidth]{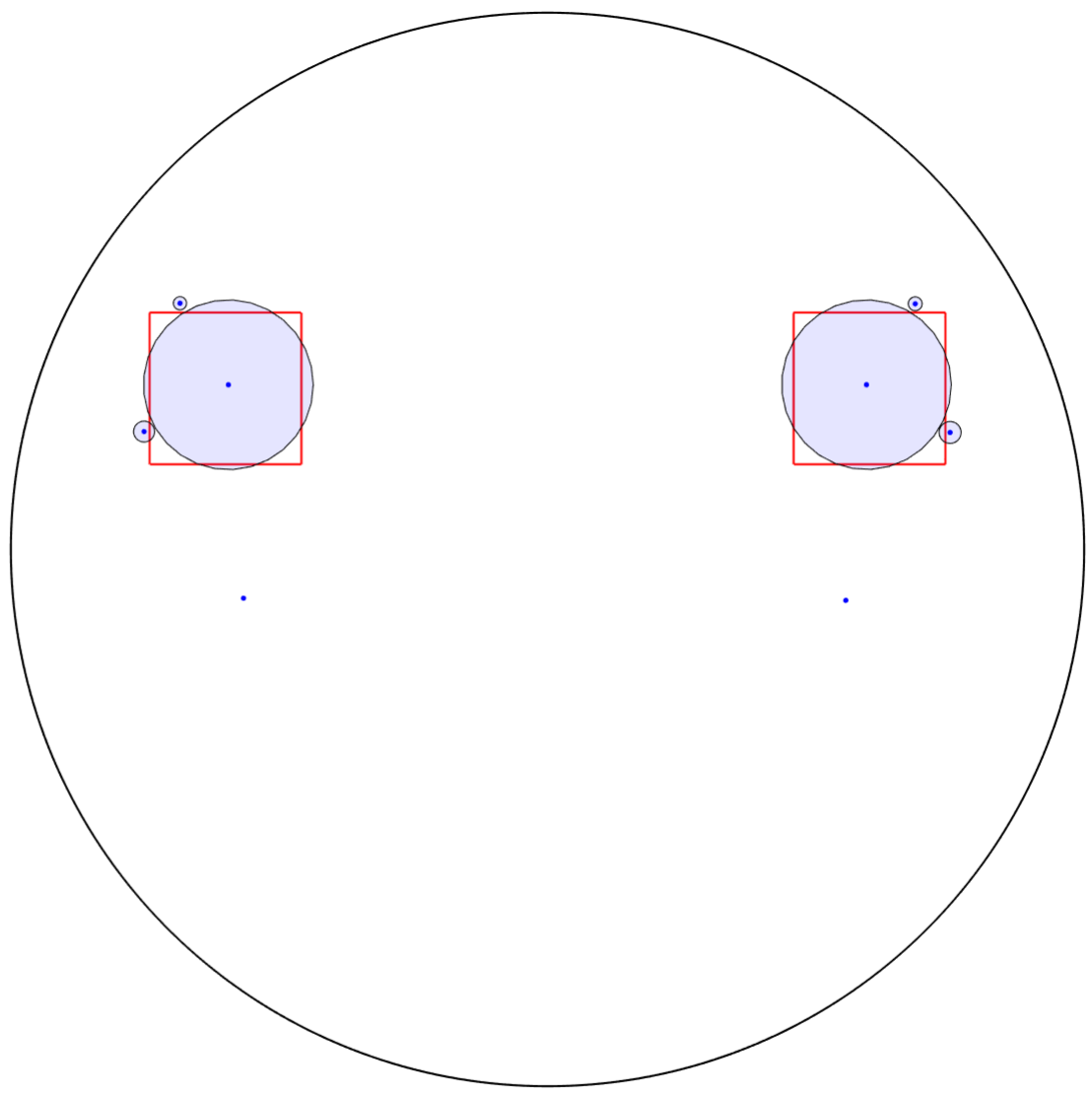}
\end{tabular}}
\caption{\label{fig9}Reconstruction of two squares with 2 and 8 disks. Small spurious disks are observed in the case $n=8$.}
\end{figure}
\section{Conclusion}
By transforming the classical inverse problem of reconstructing immersed obstacles using boundary measurements 
into a shape-from-moment problem, we have been able to develop algorithms leading to remarkably accurate reconstructions in this context. 
The methods we propose are not iterative and require no a priori information about the obstacles. A number of legitimate points remain to be addressed in a future paper:
\begin{enumerate}
\item Stability issues. The use of noisy data has not been rigorously tested in this work, but  some simple principles can nevertheless be stated.  The reconstruction methods are based on a certain number of measurements. {\it The higher the number, the more accurate the reconstruction, but also the more sensitive to noise}. For example, the reconstruction of Fig.~\ref{fig8} with only two disks can be achieved with 
9 measurements, and the reconstruction is quite stable. However, this needs to be precisely quantified. 
\item Incomplete data. One may want to take measurements only along part of the outer boundary $\varGamma_0$ and Theorem~\ref{theo:ident} states that this is theoretically sufficient to identify the obstacles. This additional difficulty has not been taken into account in this work and deserves to be studied further.
\end{enumerate}
\begin{appendix}
\section{The Dirichlet-to-Neumann operator}
\label{SEC:DtN}
Let $\varOmega$ be a $\mathcal C^{0,1}$ bounded domain and denote by $\varGamma$ its boundary. 
For every $p\in H^{1/2}(\varGamma)$, let $u_p$ be the unique function achieving:
$$\min\big\{\|\nabla \theta\|_{L^2(\varOmega)}\,:\, \theta\in H^1(\varOmega),\,\gamma_\varGamma^d\theta=p\big\}.$$
The Dirichlet-to-Neumann operator is the operator $D_\varOmega:H^{1/2}(\varGamma)\longrightarrow H^{-1/2}(\varGamma)$, $D_\varOmega p=\gamma_\varGamma^n u_p$.
Since, for every $p,q\in H^{1/2}(\varGamma)$:
$$\langle D_\varOmega p,q\rangle=-\big(\nabla u_p,\nabla u_q\big)_{L^2(\varOmega)},$$
the operator $D_\varOmega$ is self-adjoint.
\par
If $\varOmega$ is $\mathcal C^{1,1}$, by elliptic regularity, $D_\varOmega$ maps continuously  $H^{3/2}(\varOmega)$ 
into $H^{1/2}(\varOmega)$.  We deduce:
\begin{prop}
\label{prop:B1}
The operator $D_\varOmega$ extends by density as a self-adjoint operator from $H^{-1/2}(\varGamma)$ to $H^{-3/2}(\varGamma)$, i.e. for 
every $p\in H^{-1/2}(\varGamma)$ and every $q\in H^{3/2}(\varGamma)$ (with obvious notations):
$$\langle D_\varOmega p,q\rangle_{-\frac32,\frac32}=\langle p,D_\varOmega q\rangle_{-\frac12,\frac12}.$$
\end{prop}
\section{The harmonic Dirichlet problem in $L^2$}
Let $\varOmega$ be a $\mathcal C^{1,1}$ bounded domain,  denote by $\varGamma$ its boundary  and define $V=H^2(\varOmega)\cap H^1_0(\varOmega)$. For any 
$q\in H^{1/2}(\varGamma)$ define $u_q$ as the unique function in $V$ achieving:
$$\min\big\{\|\Delta\theta\|_{L^2(\varOmega)}\,:\,\gamma_\varGamma^n\theta=q\big\}.$$
Then define the operator $T:H^{1/2}(\varGamma)\longrightarrow H^{-1/2}(\varGamma)$, $Tq=\gamma_\varGamma^d\Delta u_q$. 
For every $q,p\in H^{1/2}(\varGamma)$, we have:
$$\langle Tq,p\rangle_{-\frac12,\frac12}=-\big(\Delta u_q,\Delta u_p\big)_{L^2(\varOmega)},$$
and therefore $T$ is an isomorphism from $H^{1/2}(\varGamma)$ onto $H^{-1/2}(\varGamma)$.
\begin{prop}
\label{PROP:L2H}
For every $p\in H^{-1/2}(\varGamma)$, there exists a unique function $v$ in $L^2(\varOmega)$, harmonic in $\varOmega$ 
and such that $\gamma_\varGamma^d v=p$.
\end{prop}
\begin{proof}
The function $v=\Delta u_{T^{-1}p}$ satisfies the conditions. To prove uniqueness, assume that $v$ is in $L^2(\varOmega)$, harmonic 
in $\varOmega$ and such that 
$\gamma_\varGamma^d v=0$. Let $u$ be the unique solution in  $H^1_0(\varOmega)$ 
to the problem $\Delta u=v$. By elliptic regularity, $u$ is in $V$ and, integrating by parts, $\|\Delta u\|_{L^2(\varOmega)}^2=0$. 
The conclusion follows.
\end{proof}
Note 
that in a $\mathcal C^{0,1}$ domain, uniqueness is not guaranteed (an example of square integrable non-zero harmonic function with vanishing Dirichlet 
boundary conditions in a $\mathcal C^{0,1}$ bounded domain is provided in \cite{Dauge:1987tg}).
%
\end{appendix}

\end{document}